\documentclass[11pt]{amsart}

\usepackage[T1]{fontenc}
\usepackage{lmodern}
\usepackage{amssymb, amsmath}
\allowdisplaybreaks
\usepackage[all]{xy}
\usepackage[inline]{enumitem}
\usepackage{hyphenat}
\usepackage{xcolor}
\usepackage{mathtools}
\usepackage[labelformat=empty]{caption}
\usepackage[left=3cm, right=3cm, top=2.8cm, bottom=2.8cm]{geometry}
\usepackage[colorlinks,linkcolor=blue,citecolor=blue,urlcolor=teal]{hyperref}
\usepackage{aliascnt}
\usepackage[nameinlink,noabbrev]{cleveref}

\usepackage{tikz-cd}
\tikzset{no body/.style={/tikz/dash pattern=on 0 off 1mm}}
\newcommand{\beq}{\begin{equation}}
\newcommand{\eeq}{\end{equation}}

\newcommand{\F}{\mathbb{F}}

\newcommand{\cA}{\mathcal{A}}
\newcommand{\cB}{\mathcal{B}}
\newcommand{\cE}{\mathcal{E}}

\newcommand{\cO}{\mathcal{O}}

\newcommand{\cC}{\mathcal{C}}
\newcommand{\cD}{\mathcal{D}}

\newcommand{\cR}{\mathcal{R}}

\newcommand{\cN}{\mathcal{N}}

\newcommand{\Cart}{\operatorname{Cart}}
\newcommand{\Hom}{\operatorname{Hom}}

\newcommand{\id}{\operatorname{id}}

\newcommand{\Map}{\operatorname{Map}}
\newcommand{\map}{\operatorname{map}}

\newcommand{\Mod}{\operatorname{Mod}}
\newcommand{\cMod}{\mathcal{M}\mathrm{od}}

\newcommand{\op}{{\operatorname{op}}}

\renewcommand{\lim}{\operatornamewithlimits{\rm{lim}}}
\newcommand{\colim}{\operatornamewithlimits{\rm{colim}}}

\newcommand{\QCoh}{\operatorname{{QCoh}}}
\newcommand{\Coh}{\operatorname{{Coh}}}

\newcommand{\Spec}{\operatorname{Spec}}

\newcommand{\Fil}{{\operatorname{Fil}}}

\renewcommand{\phi}{{\varphi}}

\newcommand{\Cat}{\mathrm{Cat}}

\renewcommand{\Pr}{\mathrm{Pr}}

\newcommand{\Aff}{\mathrm{Aff}}

\newcommand{\Ch}{\mathrm{Ch}}

\newcommand{\CAlg}{\mathrm{CAlg}}

\DeclareMathOperator{\Fun}{\mathrm{Fun}}

\newcommand{\LEq}{\mathrm{LEq}}

\newcommand{\ev}{\mathrm{ev}}

\newcommand{\cofib}{\operatorname{cofib}}

\newcommand{\Pro}{\operatorname{Pro}}
\newcommand{\Ind}{\operatorname{Ind}}

\newcommand{\Catbig}{\widehat{\mathrm{Cat}}}

\newcommand{\shHom}{\mathcal Hom}
\newcommand{\RHom}{\mathrm R\mathcal Hom}
\newcommand{\Crys}{\operatorname{Crys}}
\newcommand{\QCrys}{\operatorname{QCrys}}
\newcommand{\Nil}{\operatorname{Nil}}
\newcommand{\LNil}{\operatorname{LNil}}
\newcommand{\Perf}{\operatorname{Perf}}
\newcommand{\Frob}{\operatorname{Frob}}
\newcommand{\Ndg}{\mathrm N_{\mathrm{dg}}}
\newcommand{\fib}{\operatorname{fib}}
\newcommand{\Dqc}{\mathcal D_{\mathrm{qc}}}
\newcommand{\Dbc}{\mathcal D^b_{\mathrm{coh}}}
\newcommand{\Db}{\mathcal D^b}
\newcommand{\KF}{\mathcal K_F}
\newcommand{\KC}{\mathcal K_C}
\newcommand{\Perv}{\operatorname{Perv}}

\theoremstyle{plain}
\newtheorem{introthm}{Theorem}

\crefname{prop}{Proposition}{Propositions}
\Crefname{prop}{Proposition}{Propositions}

\newaliascnt{proposition}{prop}
\newtheorem{proposition}[proposition]{Proposition}
\aliascntresetthe{proposition}
\crefname{proposition}{Proposition}{Propositions}
\Crefname{proposition}{Proposition}{Propositions}

\newaliascnt{lemma}{prop}
\newtheorem{lemma}[lemma]{Lemma}
\aliascntresetthe{lemma}
\crefname{lemma}{Lemma}{Lemmas}
\Crefname{lemma}{Lemma}{Lemmas}

\newaliascnt{Sublemma}{prop}

\aliascntresetthe{Sublemma}
\crefname{Sublemma}{Sublemma}{Sublemmas}
\Crefname{Sublemma}{Sublemma}{Sublemmas}

\newaliascnt{corollary}{prop}
\newtheorem{corollary}[corollary]{Corollary}
\aliascntresetthe{corollary}
\crefname{corollary}{Corollary}{Corollaries}
\Crefname{corollary}{Corollary}{Corollaries}

\newaliascnt{theorem}{prop}
\newtheorem{theorem}[theorem]{Theorem}
\aliascntresetthe{theorem}
\crefname{theorem}{Theorem}{Theorems}
\Crefname{theorem}{Theorem}{Theorems}

\theoremstyle{definition}
\newaliascnt{definition}{prop}
\newtheorem{definition}[definition]{Definition}
\aliascntresetthe{definition}
\crefname{definition}{Definition}{Definitions}
\Crefname{definition}{Definition}{Definitions}

\newaliascnt{construction}{prop}
\newtheorem{construction}[construction]{Construction}
\aliascntresetthe{construction}
\crefname{construction}{Construction}{Constructions}
\Crefname{construction}{Construction}{Constructions}

\newaliascnt{notation}{prop}

\aliascntresetthe{notation}

\newaliascnt{example}{prop}

\aliascntresetthe{example}

\newaliascnt{rmk}{prop}
\newtheorem{rmk}[rmk]{Remark}
\aliascntresetthe{rmk}
\crefname{rmk}{Remark}{Remarks}
\Crefname{rmk}{Remark}{Remarks}

\newaliascnt{remark}{prop}
\newtheorem{remark}[remark]{Remark}
\aliascntresetthe{remark}
\crefname{remark}{Remark}{Remarks}
\Crefname{remark}{Remark}{Remarks}

\Crefname{introthm}{Theorem}{Theorems}

\numberwithin{equation}{section}
\setlist[enumerate]{label=\textup{(\arabic*)},leftmargin=*}
\setlist[enumerate,2]{label=\textup{(\alph*)},leftmargin=*}

\title{Frobenius--Cartier duality via lax equalizers}
\author{Fei Ren}
\address{Bergische Universit\"at Wuppertal, Gau\ss str.\ 20, D-42119 Wuppertal, Germany}
\email{renfei@uni-wuppertal.de}
\date{7 September 2026}
\begin{document}
\begin{abstract}
We prove a duality theorem for lax equalizers of stable $\infty$-categories, and we apply it to Frobenius and Cartier modules on a Noetherian $F$-finite scheme $X$ of characteristic $p$ with a unit dualizing complex.
When $X$ has affine diagonal, comparison with the derived categories of coherent modules and passage to homotopy categories recovers Baudin's derived Hom duality.
The duality of coherent modules exchanges the standard $t$-structure on one side with the perverse $t$-structure determined by the dualizing complex on the other.
In particular, it identifies the abelian category of coherent Cartier modules with a perverse heart of coherent Frobenius modules, and conversely.
\end{abstract}
\maketitle
\tableofcontents

\section{Introduction}
Fix a prime $p$. For an $\F_p$-scheme $X$, write $F:X\to X$ for absolute Frobenius.
\begin{definition}
A \textit{Frobenius module} is a quasi-coherent sheaf $M$ with a map $\tau:M\to F_*M$; a \textit{Cartier module} is a quasi-coherent sheaf $N$ with a map $\kappa:F_*N\to N$.
We write $\Coh^F_X$ and $\Coh^C_X$ for the abelian categories whose underlying sheaves are coherent.
\end{definition}
Frobenius modules and Cartier modules have been studied for a long time on each side separately.
Lyubeznik's unit $F$-modules \cite{Lyu97}, Emerton--Kisin's Riemann--Hilbert correspondence for unit $F$-crystals \cite{EK}, and B\"ockle--Pink's crystals \cite{BP} treat the Frobenius side.
Blickle--B\"ockle developed Cartier modules and crystals and proved strong finiteness results \cite[Theorem 4.6; Corollary 4.7]{BB}, with further derived and functorial developments in \cite{BB13}.
On smooth schemes the two sides are linked by Grothendieck--Serre duality, or by twisting by a dualizing sheaf: see \cite[Theorem 5.9; Theorem 5.15]{BB} and, for unit modules on regular rings, \cite[Theorem 2.5]{Ma14}.
Schedlmeier and Ohkawa extended Riemann--Hilbert-type equivalences with constructible or perverse $\F_p$-sheaves to many singular embeddable schemes \cite[Theorem 4.12; Corollary 4.14]{Sch18}, \cite[Theorem 4.12; Theorem 5.5]{Ohk18}.
Bhatt--Lurie later gave a broad Riemann--Hilbert functor into Frobenius modules \cite[Theorem 1.0.2; Theorem 12.1.5]{BL}.
Baudin then constructed an explicit derived Hom duality between Frobenius and Cartier complexes using a \textit{unit} dualizing complex \cite[Theorem 4.2.7]{Bau}, passed it to crystals \cite[Theorem 4.3.5]{Bau}, and recovered a duality with perverse constructible sheaves without an embedding hypothesis \cite[Theorem 5.2.7]{Bau}.
Here ``unit'' refers to the adjoint Cartier map $\omega\simeq F^!\omega$, rather than to invertibility of $F_*\omega\to\omega$ \cite[Definition 4.2.1; Remark 4.2.2]{Bau}.

The aim of this article is to study the duality phenomenon between Frobenius and Cartier modules via lax equalizers.
Write $\Dbc(X)$ for the full subcategory of the Zariski limit $\Dqc(X)$ of \Cref{def:dqc} on objects whose affine restrictions lie in $\mathcal D^b_{\mathrm{coh}}(R)$.
Finite duality and a chosen equivalence $u:\omega\simeq F^!\omega$ give a natural equivalence
\begin{equation}\label{eq:intro-gamma}
 \gamma:F_*D_\omega\simeq
                  D_\omega F_*^{\op},\qquad
 D_\omega=\RHom(-,\omega).
\end{equation}
This is an equivalence in a functor $\infty$-category, constructed in \Cref{prop:finite-duality}.
Using only the underlying dualizing equivalence $D_\omega=\RHom(-,\omega)$ on $\Dbc(X)$ and the natural equivalence $\gamma$, one obtains an equivalence of lax equalizers $\Frob(\Dbc(X),F_*)^{\op}\simeq\Cart(\Dbc(X),F_*)$ on $\Dbc(X)$.

\begin{introthm}[{\Cref{thm:intrinsic}}]\label{thm:intro-leq}
Let $X$ be Noetherian and $F$-finite, with a dualizing complex $\omega$ and a chosen equivalence $u:\omega\simeq F^!\omega$.
There is an exact equivalence
\begin{equation}\label{eq:intro-leq}
 \Frob(\Dbc(X),F_*)^{\op}\simeq\Cart(\Dbc(X),F_*).
\end{equation}
The Cartier map of the image of $(M,\tau)$ is
\[
 F_*D_\omega M\xrightarrow{\gamma_M}
 D_\omega(F_*M)\xrightarrow{D_\omega(\tau)}
 D_\omega M.
\]
On objects, a quasi-inverse sends $(N,\kappa:F_*N\to N)$ to $\bigl(D_\omega N,\ \sigma\bigr)$, where the Frobenius map of the image is
\[
 \sigma:
 D_\omega N
 \xrightarrow{D_\omega(\kappa)}
 D_\omega(F_*N)
 \xrightarrow{\gamma_N^{-1}}
 F_*D_\omega N.
\]
\end{introthm}

\begin{introthm}[{\Cref{thm:bidual-leq}}]\label{thm:intro-bidual}
In the situation of \Cref{thm:intro-leq}, write $D_\omega^{\mathrm{LEq}}$ for the equivalence \eqref{eq:intro-leq} and write $\eta$ for the unit of tensor-Hom $\mathrm{id}\simeq D_\omega\circ D_\omega^{\op}$ on $\Dbc(X)$.
There are canonical equivalences
\[
 (D_\omega^{\mathrm{LEq}})^{\op}\circ D_\omega^{\mathrm{LEq}}
 \simeq
 \mathrm{id}_{\Frob(\Dbc(X),F_*)}
 \qquad\text{and}\qquad
 D_\omega^{\mathrm{LEq}}\circ (D_\omega^{\mathrm{LEq}})^{\op}
 \simeq
 \mathrm{id}_{\Cart(\Dbc(X),F_*)}
\]
in the respective functor $\infty$-categories, with underlying arrow $\eta$ on the Frobenius side.
A quasi-inverse of $D_\omega^{\mathrm{LEq}}$ is realized by reversing the lax equalizer a second time, using the mate of $\gamma^{-1}$ under $\eta$.
\end{introthm}

\begin{remark}\label{rem:intro-no-baudin}
The results of Theorems \ref{thm:intro-leq} and \ref{thm:intro-bidual} are independent of Baudin’s work \cite{Bau}. Baudin works exclusively with triangulated categories of Frobenius and Cartier complexes, and therefore has no direct analogue of the lax-equalizer statements. When the diagonal is affine (for instance when $X$ is separated; see also Remark \ref{rem:affine-diagonal}), Theorem \ref{thm:comparison} identifies \eqref{eq:intro-leq} with the module duality of Theorem \ref{thm:modules}. Passing to homotopy categories then recovers Baudin’s derived Hom duality \cite[Theorem 4.2.7]{Bau}. Finally, the double-dual equivalences of Theorem \ref{thm:intro-bidual} recover the involutivity of Baudin’s Hom duality.
\end{remark}

We next enhance Baudin's Hom duality on module categories, and pass to crystals.
Let $\Db$ be the bounded derived $\infty$-category of \Cref{def:bounded}.
Let $\cN_F(X)\subset\Db(\Coh^F_X)$ and $\cN_C(X)\subset\Db(\Coh^C_X)$ be the full subcategories of complexes with nilpotent cohomology.

\begin{introthm}[{\Cref{thm:modules}}]\label{thm:intro}
Let $X$ be Noetherian and $F$-finite, and let $\omega$ be a unit dualizing complex.
There is an exact equivalence
\[
 \mathbb D_\omega:\Db(\Coh^F_X)^{\op}\xrightarrow{\ \simeq\ }\Db(\Coh^C_X)
\]
whose homotopy functor agrees with Baudin's derived Hom duality \cite[Theorem 4.2.7]{Bau}.
\end{introthm}
\begin{rmk}
We wish to stress here that  \Cref{thm:intro} is dependent on the work of Baudin: its proof passes through \Cref{lem:exact-detection} and ultimately reduces to a computation carried out in the triangulated setting, namely \cite[Theorem 4.2.7]{Bau}.
\end{rmk}

\begin{introthm}[{\Cref{thm:crystals}}]\label{thm:intro-crystals}
In the situation of \Cref{thm:intro}, the functor $\mathbb D_\omega$ induces an exact equivalence
\[
 \bigl(\Db(\Coh^F_X)/\cN_F(X)\bigr)^{\op}
       \simeq\Db(\Coh^C_X)/\cN_C(X)
\]
whose homotopy functor agrees with Baudin's duality of crystals \cite[Theorem 4.3.5]{Bau}.
There are exact equivalences
\begin{align*}
 \Db(\Coh^F_X)/\cN_F(X)
       &
       \simeq\Db(\Crys^F_X),\\
 \Db(\Coh^C_X)/\cN_C(X)
       &
       \simeq\Db(\Crys^C_X).
\end{align*}
\end{introthm}

When $X$ has affine diagonal, the derived comparison theorems of Mattis--Wei\ss\ identify \eqref{eq:intro-leq} with the module equivalence of \Cref{thm:intro}; see \Cref{thm:comparison}.

\begin{introthm}[{\Cref{thm:comparison}}]\label{thm:intro-comparison}
Assume in addition that $X$ has affine diagonal.
There are exact equivalences, commuting with the underlying-complex functors,
\begin{align*}
 \Db(\Coh^F_X)&\simeq\Frob(\Dbc(X),F_*),\\
 \Db(\Coh^C_X)&\simeq\Cart(\Dbc(X),F_*).
\end{align*}
With the unit data induced by an injective representative of $\omega$, these equivalences identify $\mathbb D_\omega$ with the lax-equalizer duality of \Cref{thm:intro-leq}.
\end{introthm}

The perverse $t$-structure is determined by $\omega$.
For $x\in X$, write $d_\omega(x)$ for the unique cohomological degree in which the local cohomology complex $\mathrm R\Gamma_{\mathfrak m_x}(\omega_x)$ is nonzero; equivalently, after shifting to a normalized dualizing complex \cite[\href{https://stacks.math.columbia.edu/tag/0AWF}{Tag 0AWF}]{SP}, $\mathrm R\Gamma_{\mathfrak m_x}(\omega_x)$ is an injective hull of $k(x)$ placed in degree $d_\omega(x)$ \cite[\href{https://stacks.math.columbia.edu/tag/0A82}{Tag 0A82}]{SP}.

\begin{introthm}[{\Cref{thm:perverse}}]\label{thm:intro-perverse}
In the situation of \Cref{thm:intro}, the equivalence $\mathbb D_\omega$ exchanges the standard $t$-structure on one module category with the perverse $t$-structure determined by $d_\omega$ on the other.
In particular there are anti-equivalences of abelian hearts
\[
 \Perv_\omega(\Coh^F_X)^{\op}\simeq\Coh^C_X,
 \qquad
 \Perv_\omega(\Coh^C_X)^{\op}\simeq\Coh^F_X.
\]
\end{introthm}
\Cref{rem:middle-counterexample} shows why the formula $d_\omega(x)=-\dim\overline{\{x\}}$ needs an extra normalization hypothesis.

\begin{introthm}[{\Cref{thm:descent}}]\label{thm:intro-descent}
Let $\{U_i\}_{i=1}^r$ be a finite Zariski open cover of a Noetherian $F$-finite scheme $X$, and write $U_\bullet$ for its \v Cech nerve.
Restriction induces equivalences
\begin{align*}
 \Dbc(X)&\simeq\lim_{[n]\in\Delta}\Dbc(U_n),\\
 \Frob(\Dbc(X),F_*)&\simeq\lim_{[n]\in\Delta}\Frob(\Dbc(U_n),F_*),\\
 \Cart(\Dbc(X),F_*)&\simeq\lim_{[n]\in\Delta}\Cart(\Dbc(U_n),F_*).
\end{align*}
For fixed unit dualizing data, the equivalence of \Cref{thm:intro-leq} commutes with the diagrams obtained by restricting along the open immersions $U_n\to X$ of the \v Cech nerve.
\end{introthm}

\begin{introthm}[{\Cref{cor:indpro-modules}}]\label{thm:intro-indpro}
An exact equivalence $D:\cC^{\op}\simeq\cD$ of small stable $\infty$-categories extends to an exact equivalence $\Pro(\cC)^{\op}\simeq\Ind(\cD)$.
This applies to the dualities of \Cref{thm:intro-leq,thm:intro,thm:intro-crystals}, and to the further dualities listed in \Cref{cor:indpro-modules}.
\end{introthm}

\begin{remark}[Proof strategy]\label{rem:intro-cost}
The technical main input behind \Cref{thm:intro-leq,thm:intro-bidual} is \Cref{prop:reversal}: an equivalence $D:\cC^{\op}\simeq\cD$ together with compatible natural equivalences $DF^{\op}\simeq F'D$ and $DG^{\op}\simeq G'D$ induces $\LEq(F,G)^{\op}\simeq\LEq(G',F')$.
Concretely, the natural equivalence $\gamma:F_*D_\omega\simeq D_\omega F_*^{\op}$ of \Cref{prop:finite-duality}, obtained from finite Grothendieck duality and a chosen unit equivalence $u:\omega\simeq F^!\omega$, is the input to which we apply \Cref{prop:reversal}.
This yields the exact equivalence
\[
 D_\omega^{\mathrm{LEq}}:\Frob(\Dbc(X),F_*)^{\op}\simeq\Cart(\Dbc(X),F_*)
\]
of \Cref{thm:intrinsic}.
The unit $\eta$ of tensor-Hom from \Cref{lem:grothendieck} gives a mate $\gamma^\vee$ of $\gamma$; a second application of \Cref{prop:reversal} yields
\[
 (D_\omega^{\mathrm{LEq}})^{\op}\circ D_\omega^{\mathrm{LEq}}
 \simeq
 \mathrm{id}
\]
as morphisms in $\Fun\bigl(\Frob(\Dbc(X),F_*),\Frob(\Dbc(X),F_*)\bigr)$ (and dually on the Cartier side); see \Cref{thm:bidual-leq}.

Under an affine diagonal, the derived comparison theorems of Mattis--Wei\ss\ identify $\Db(\Coh^F_X)$ and $\Db(\Coh^C_X)$ with $\Frob(\Dbc(X),F_*)$ and $\Cart(\Dbc(X),F_*)$, respectively (\Cref{thm:comparison}), and thereby identify $D_\omega^{\mathrm{LEq}}$ with the module duality $\mathbb D_\omega$ of \Cref{thm:modules}.
\Cref{thm:modules} constructs $\mathbb D_\omega$ as an exact $\infty$-functor on $\Db(\Coh^F_X)$ and $\Db(\Coh^C_X)$; its homotopy functor is Baudin's derived Hom duality \cite[Theorem 4.2.7]{Bau}, and \Cref{lem:exact-detection} upgrades that triangulated equivalence to an equivalence of stable $\infty$-categories.
For crystals the localization at nilpotent cofiber cohomology is a stable Verdier quotient \cite[Theorem I.3.3(i)]{NS}. An abstract internal Hom pairing is \Cref{prop:abstract-hom}.
\end{remark}

\subsection{Conventions and notations}\label{subsec:conventions}
We fix universes $\mathbb U\in\mathbb V$.
We write $\Cat_\infty$ for the $\infty$-category of $\mathbb U$-small $\infty$-categories and $\Catbig_\infty$ for the larger one, as in \cite[Definition 3.0.0.1; Remark 3.0.0.5]{HTT}.
Abelian categories used as input to $\Db(-)$ are assumed $\mathbb U$-small.
An $\infty$-category is essentially small if and only if it is equivalent to a small $\infty$-category \cite[Definition 5.4.1.3; Proposition 5.4.1.2(3)]{HTT}.
Essentially small bounded $\infty$-categories are replaced by small equivalent $\infty$-categories when forming Ind- and Pro-completions.

All limits, localizations and mapping spaces below are $\infty$-categorical.
We use cohomological grading.
$h\cC$ denotes the homotopy category of $\cC$.
If $\cC$ is an $\infty$-category, $\Map_{\cC}(X,Y)$ denotes the mapping space.
If $\cC$ is stable, $\map_{\cC}(X,Y)$ denotes the mapping spectrum, so that
\[
 \Omega^\infty(\map_{\cC}(X,Y)[n])\simeq\Map_{\cC}(X,Y[n]).
\]
\cite[Definition 2.15]{BGT} via \cite[Notation 1.4.2.20]{HA}.
We use $\Mod_R, \Mod(\cO_X)$ for the abelian categories of modules over the corresponding rings, while we use $\cMod_\cR$ for the stable $\infty$-category over a ring spectrum $\cR$.

\section{Stable derived categories and lax equalizers}\label{sec:foundations}
We use the conventions of \Cref{subsec:conventions}.
Many of the results in \Cref{sec:foundations} and \Cref{sec:finite} are of course well known in the triangulated setting, but since we could not find a satisfactory reference for them in the $\infty$-categorical framework, we take some care to record them here.

\begin{definition}[{Stable Verdier quotient, \cite[Definition 1.1, Definition 3.1(ii)]{Drew}}]\label{def:bounded}
Let $\cA$ be a $\mathbb U$-small abelian category.
Let $\Ch^b(\cA)$ be the differential graded category of bounded cochain complexes and let $\operatorname{Ac}^b(\cA)$ be its full subcategory of acyclic complexes.
Set
\[
 \Db(\cA):=\Ndg\Ch^b(\cA)\big/\Ndg\operatorname{Ac}^b(\cA).
\]
\end{definition}

\begin{lemma}[{\cite[Definition 3.1; Theorem 1.3]{Drew}}]\label{lem:bounded-stable}
Let $\cA$ be a $\mathbb U$-small abelian category.
Then:
\begin{enumerate}
\item
$\Db(\cA)$ is stable.
\item
If $W$ denotes quasi-isomorphisms in the ordinary category of bounded complexes and chain maps, there is an equivalence of $\infty$-categories
\begin{equation}\label{eq:relative-model}
 \Db(\cA)\simeq\mathrm N(\Ch^b(\cA))[W^{-1}].
\end{equation}
\item
The homotopy category $h\Db(\cA)$ is the usual bounded derived category of $\cA$.
\end{enumerate}
\end{lemma}
\begin{proof}
Since the citations in \cite{Drew} are based on an earlier version of \cite{HA}, we recall the comparison used here.
Bounded complexes are closed under shifts and mapping cones, so their dg nerve is stable \cite[Proposition 1.3.2.10; Remark 1.3.2.17]{HA}.
The ordinary nerve localized at chain homotopy equivalences is this dg nerve by \cite[Proposition 1.3.4.5]{HA}: the hypothesis of that proposition is that for every object $M_\bullet$ of the chosen full subcategory of chain complexes, the mapping cylinder $N_*(\Delta^1)\otimes M_\bullet$ of $\id_{M_\bullet}$ still lies in the subcategory.
For bounded complexes the mapping cylinder is obtained by a finite sum with a shift, hence remains bounded.
Localizing further at quasi-isomorphisms is the Verdier quotient of \Cref{def:bounded} by \cite[Theorem 1.3]{Drew}.
The homotopy category of an $\infty$-localization is the localization of the homotopy category \cite[\href{https://kerodon.net/tag/01MV}{Tag 01MV}]{Kerodon}, which is the usual bounded derived category.
\end{proof}

\begin{lemma}\label{lem:grothendieck-derived}
Let $\cB$ be a Grothendieck abelian category.
The derived $\infty$-category $\cD(\cB)$ is stable and presentable.
Its standard $t$-structure has heart $\cB$.
\end{lemma}
\begin{proof}
Stability of $\cD(\cB)$ is \cite[Proposition 1.3.5.9]{HA}.
Presentability and the standard $t$-structure are \cite[Proposition 1.3.5.21(1)(2)]{HA}.
The heart of that $t$-structure is equivalent to $\cB$ by \cite[Remark 1.3.5.19]{HA}.
\end{proof}

\begin{lemma}\label{lem:bounded-serre}
Let $\cB$ be a Grothendieck abelian category and let $\cA\subset\cB$ be a Serre subcategory.
For $*\in\{+,-,b\}$, write $\mathcal D^*_{\cA}(\cB)$ for the full subcategory of $\cD(\cB)$ consisting of objects whose cohomology lies in $\cA$ and vanishes outside the usual range indicated by $*$.
Then:
\begin{enumerate}
\item
Each $\mathcal D^*_{\cA}(\cB)$ is stable.
\item
The homotopy category $h\mathcal D^*_{\cA}(\cB)$ is the usual triangulated subcategory $\mathrm D^*_{\cA}(\cB)\subset\mathrm D(\cB)$.
\end{enumerate}
\end{lemma}
\begin{proof}
The ambient category $\cD(\cB)$ is stable by \Cref{lem:grothendieck-derived}.
A cofiber may be computed by a localized mapping cone \cite[Remark 1.3.2.17; Proposition 1.3.5.13]{HA}.
The associated cohomology sequence is exact in $\cB$.
The Serre condition keeps those cohomology objects in $\cA$.
Shifts and finite cofibers preserve the boundedness conditions encoded by $*$.
A full subcategory of a stable $\infty$-category closed under shifts and finite cofibers is stable \cite[Lemma 1.1.3.3]{HA}.
This proves (1).
For (2): $\mathcal D^*_{\cA}(\cB)$ is a full subcategory of the stable $\infty$-category $\cD(\cB)$, so its homotopy category is the full subcategory of $h\cD(\cB)$ on the same objects, which is the classical triangulated subcategory $\mathrm D^*_{\cA}(\cB)$.
\end{proof}

\begin{lemma}[{\cite[Proposition 5.10; Corollary 5.11]{BGT}}]\label{lem:exact-detection}
Let $G:\cC\to\cD$ be an exact functor between stable $\infty$-categories.
If $hG$ is fully faithful, then $G$ is fully faithful.
If $hG$ is an equivalence, then $G$ is an equivalence.
\end{lemma}
\begin{proof}
Exactness implies that $G$ preserves zero objects and cofibres \cite[Proposition 1.1.4.1]{HA}, hence $G(M[n])\simeq(GM)[n]$ for all $n\in\mathbb Z$.
The induced maps of mapping spaces $\Map_{\cC}(M,N)\to\Map_{\cD}(GM,GN)$ are pointed by the zero maps.
On $\pi_0$ they are $\Hom_{h\cC}(M,N)\to\Hom_{h\cD}(GM,GN)$, which are bijections if $hG$ is fully faithful.
Stability gives $\Map_{\cC}(M[n],N)\simeq\Omega^n\Map_{\cC}(M,N)$ \cite[Remark 1.1.2.8]{HA}, so
\[
 \pi_n\Map_{\cC}(M,N)\simeq\Hom_{h\cC}(M[n],N)
\]
for $n\ge 1$ \cite[Notation 1.1.2.17]{HA}, and likewise in $\cD$.
The maps on $\pi_n$ are therefore identified with
\[
 \Hom_{h\cC}(M[n],N)\longrightarrow\Hom_{h\cD}(G(M[n]),GN),
\]
which are bijections by fully faithfulness of $hG$ together with $G(M[n])\simeq(GM)[n]$.
Thus $G$ is fully faithful.

If $hG$ is an equivalence, it is essentially surjective.
A morphism of an $\infty$-category is an equivalence if and only if it is an isomorphism in the homotopy category \cite[\href{https://kerodon.net/tag/004Q}{Tag 004Q}]{Kerodon}, so $G$ is essentially surjective as well, hence an equivalence.
\end{proof}

A localization $L:\cC\to\cD$ of $\infty$-categories is \textit{reflective} if it admits a fully faithful right adjoint \cite[Definition 5.2.7.2]{HTT}.
Equivalently, $L$ is a localization whose right adjoint is fully faithful, so that $\cD$ is identified with a reflective subcategory of $\cC$.

\begin{lemma}\label{lem:reflective-quotient}
Let $F:\cA\to\cB$ be an exact quotient functor of Grothendieck abelian categories, and suppose $F$ admits a fully faithful right adjoint $G$.
Then the derived adjunction $\cD(F)\dashv\mathrm RG$ is a reflective localization of stable $\infty$-categories.
\end{lemma}
\begin{proof}
In a Grothendieck abelian category every object has a bounded below resolution by injectives \cite[Corollary 1.3.5.7]{HA}.
For such a complex $I$, the complex $\Hom^\bullet(A,I)$ is acyclic whenever $A$ is acyclic, so $\mathrm{RHom}(-,I)$ may be computed by the underived Hom complex into $I$ \cite[\href{https://stacks.math.columbia.edu/tag/070J}{Tag 070J}]{SP}.
Thus $\mathrm RG$ may be computed by applying $G$ degreewise to a bounded below injective resolution.

If $I$ is a bounded below complex of injectives in $\cB$, then so is $GI$ in $\cA$: one has $\Hom^\bullet(A,GI)=\Hom^\bullet(FA,I)$, and $F$ preserves acyclic complexes, so the same Tag~070J criterion applies to $GI$.
The underived counit $\varepsilon:FG\to\mathrm{id}$ is an isomorphism on injectives (since $FGI=I$ termwise), hence induces an equivalence on $\cD(\cB)$ after deriving.

The injective model structure on chain complexes in a Grothendieck abelian category is given by \cite[Proposition 1.3.5.3]{HA}.
The $\infty$-categorical localization at quasi-isomorphisms is given by \cite[Proposition 1.3.5.15]{HA}.
Those two propositions produce the stable $\infty$-categories $\cD(\cA)$ and $\cD(\cB)$, not yet the adjunction between them.

For $I$ a bounded below complex of injectives, the natural isomorphism of mapping complexes gives an equivalence of derived mapping spaces $\Map(A,\mathrm RG\,I)\simeq\Map(FA,I)$.
This uses the dg-nerve comparison of \cite[Remark 1.3.1.12]{HA}, together with \cite[Proposition 1.3.1.17]{HA} and \cite[Proposition 1.3.5.13]{HA}.

Together with the induced unit, one obtains an adjunction of $\infty$-categories by \cite[Proposition 5.2.2.8]{HTT}.
The counit of this adjunction is an equivalence, as recorded above.
For an adjunction $L\dashv R$ of $\infty$-categories, $R$ is fully faithful if and only if the counit is an equivalence: fully faithfulness means that $\Map(X,Y)\to\Map(RX,RY)$ is an equivalence for all $X,Y$ \cite[Definition 1.2.10.1]{HTT}, and the triangle identities identify this map with precomposition against the counit (equivalently, \cite[\href{https://kerodon.net/tag/02FF}{Tag 02FF}]{Kerodon}).
Hence $\mathrm RG$ is fully faithful, and the localization $\cD(F)$ is reflective.
\end{proof}

\begin{proposition}\label{prop:realization}
For a Noetherian $F$-finite $X$, termwise inclusion induces exact equivalences
\begin{align*}
 J_F:\Db(\Coh^F_X)&\simeq\mathcal D^b_{\mathrm{coh}}(\QCoh^F_X),\\
 J_C:\Db(\Coh^C_X)&\simeq\mathcal D^b_{\mathrm{coh}}(\QCoh^C_X).
\end{align*}
\end{proposition}
\begin{proof}
Termwise inclusion preserves quasi-isomorphisms and mapping cones.
By \cite[Proposition 1.3.5.13]{HA}, the stable $\infty$-categories $\mathcal D^b_{\mathrm{coh}}(\QCoh^F_X)$ and $\mathcal D^b_{\mathrm{coh}}(\QCoh^C_X)$ may be obtained by localizing the dg nerves of bounded coherent complexes of quasi-coherent sheaves at quasi-isomorphisms; \cite[Proposition 1.3.5.15]{HA} records the corresponding localization presentations of the unbounded derived $\infty$-categories.
Combined with \Cref{lem:bounded-stable}, termwise inclusion therefore defines exact $\infty$-functors from $\Db(\Coh^F_X)$ and $\Db(\Coh^C_X)$ into those localized $\infty$-categories.
Its image has bounded coherent cohomology.
The homotopy functors of $J_F$ and $J_C$ are precisely the classical comparison functors in \cite[Proposition 2.2.15; Corollary 3.2.24]{Bau}, respectively.
Those comparison functors are equivalences; \Cref{lem:exact-detection} applies to the exact functors just constructed.
\end{proof}

\subsection{Lax equalizers}
\begin{definition}[{Lax equalizer, \cite[Definition II.1.4]{NS}}]\label{def:leq}
For functors $F,G:\cC\to\cD$, define
\[
\begin{tikzcd}
 \LEq(F,G)\ar[r]\ar[d] & (\cD)^{\Delta^1}\ar[d,"{(\ev_0,\ev_1)}"]\\
 \cC\ar[r,"{(F,G)}"']&\cD\times\cD
\end{tikzcd}
\]
as a pullback in $\Catbig_\infty$.
Following \cite[Definition 2.4]{MWC}, for an endofunctor $H:\cC\to\cC$ we set
\[
 \Frob(\cC,H)=\LEq(\id,H)
 \qquad\text{and}\qquad
 \Cart(\cC,H)=\LEq(H,\id).
\]
\end{definition}
An object of $\LEq(F,G)$ is described by $(M,a:FM\to GM)$.
A morphism is an underlying map $f:M\to N$ together with a path between $G(f)a$ and $bF(f)$ in $\Map_{\cD}(FM,GN)$.
If $\cC$ and $\cD$ are stable and $F,G$ are exact, then $\LEq(F,G)$ is stable and the forgetful functor to $\cC$ is exact \cite[Proposition 2.6(e)]{MWC}.
There is a natural equivalence of mapping spaces
\begin{equation}\label{eq:map-fiber}
 \Map_{\LEq(F,G)}((M,a),(N,b))\simeq
 \fib\left(\Map_{\cC}(M,N)
 \xrightarrow{\,G(-)a-bF(-)\,}\Map_{\cD}(FM,GN)\right),
\end{equation}
and likewise for mapping spectra, by \cite[Proposition 2.6(a)]{MWC}.

\begin{proposition}\label{prop:reversal}
Suppose $D:\cC^{\op}\simeq\cD$ is an equivalence and $F,G$ and $F',G'$ are endofunctors of $\cC$ and $\cD$, respectively.
Given natural equivalences
\[
 DF^{\op}\simeq F'D,\qquad DG^{\op}\simeq G'D
\]
in $\Fun(\cC^{\op},\cD)$, there is an equivalence
\[
 \LEq(F,G)^{\op}\simeq\LEq(G',F').
\]
\end{proposition}
\begin{proof}
The functor $(-)^{\op}:\Catbig_\infty\to\Catbig_\infty$ is an involution of $\infty$-categories.
An equivalence of $\infty$-categories preserves both limits and colimits of diagrams.
In particular $(-)^{\op}$ preserves pullbacks in $\Catbig_\infty$.\footnote{This is distinct from the relation $(\lim H)^{\op}\simeq\colim H^{\op}$ for a diagram $H$ taking values in a single $\infty$-category: that relation concerns limits \emph{in} $\cC$, whereas here we take limits \emph{of} $\infty$-categories.}

Reversing an arrow identifies
$\bigl((\cC)^{\Delta^1}\bigr)^{\op}\simeq(\cC^{\op})^{\Delta^1}$, with the two evaluations exchanged.
Applying these operations to \Cref{def:leq} gives
$\LEq(F,G)^{\op}\simeq\LEq(G^{\op},F^{\op})$.
The functor $D$ and the specified natural equivalences give an equivalence between the cospan defining the latter pullback and the cospan defining $\LEq(G',F')$.
Taking pullbacks of these equivalent diagrams proves the assertion.
\end{proof}

\begin{remark}[Adjunction vs duality]
An adjunction $L\dashv R$ gives a covariant equivalence $\Cart(\cC,L)\simeq\Frob(\cC,R)$ \cite[Proposition 2.9; Corollary 2.10]{MWC}.
For Frobenius this involves $\mathrm LF^*\dashv F_*$ on the derived category.
It does not itself reverse arrows or require a dualizing complex and is distinct from \Cref{prop:reversal}.
\end{remark}

\begin{proposition}[Internal Hom pairing]\label{prop:abstract-hom}
Let $\cC$ be a closed monoidal $\infty$-category, and let
$C\mapsto B^C$ denote the right adjoint to $A\mapsto A\otimes B$,
as in \cite[Definition 4.1.1.15; Remark 4.1.1.16]{HA}.
Let $G:\cC\to\cC$ carry a lax monoidal $\infty$-functor structure
in the sense of \cite[Definition 2.1.2.7; Example 2.2.6.10]{HA},
with structure map $\mu_{A,B}:GA\otimes GB\to G(A\otimes B)$.
Then:
\begin{enumerate}
\item
There is a natural transformation
\[
 \chi_{B,C}:G B^C\longrightarrow (GB)^{GC}
\]
in $\Fun(\cC^{\op}\times\cC,\cC)$.
\item
It induces a functor
\[
 \Frob(\cC,G)^{\op}\times\Cart(\cC,G)
       \longrightarrow\Cart(\cC,G)
\]
whose value on $(B,\tau:B\to GB)$ and $(C,\kappa:GC\to C)$
has underlying object $B^C$ and structure map
\begin{equation}\label{eq:abstract-hom-structure}
 G B^C\xrightarrow{\chi_{B,C}}(GB)^{GC}
                \xrightarrow{\tau^\kappa}B^C.
\end{equation}
\end{enumerate}
\end{proposition}
\begin{proof}
For independent variables $A,B,C$, closedness and the lax tensor map give
\begin{align*}
 \Map(A,B^C)
 &\simeq\Map(A\otimes B,C)\\
 &\longrightarrow\Map(G(A\otimes B),GC)\\
 &\longrightarrow\Map(GA\otimes GB,GC)\\
 &\simeq\Map(GA,(GB)^{GC}).
\end{align*}
The first middle arrow is the map on mapping spaces induced by the $\infty$-functor $G$; the second
precomposes with $\mu_{A,B}$.
This composite is natural on
$\cC^{\op}\times\cC^{\op}\times\cC$.

Write $y$ for the Yoneda embedding and $G^*$ for precomposition by $G$ on presheaves, enlarging the universe if necessary.
The above composite is a natural transformation
\[
 yB^C\longrightarrow G^*y (GB)^{GC}
\]
of presheaves in the independent variable $A$.
This is a transformation of presheaf-valued functors of $(B,C)$.
The functor $\infty$-category of presheaves is presentable, and precomposition $G^*$ preserves limits (they are computed pointwise), so $G^*$ has a left adjoint $G_!$ by the adjoint functor theorem \cite[Corollary 5.5.2.9]{HTT}.
For $P$ a presheaf, the adjunction and the Yoneda lemma give
$\Map(G_!y(X),P)\simeq (G^*P)(X)=P(GX)\simeq\Map(y(GX),P)$,
hence $G_!y\simeq yG$
\cite[Proposition 5.1.3.1]{HTT}.
Taking the mate in the functor category with parameters $(B,C)$
therefore gives a natural transformation
\[
 yG B^C\longrightarrow y (GB)^{GC}.
\]
Full faithfulness of the Yoneda embedding yields the natural transformation $\chi$.
Under the tensor-Hom adjunction, $\chi_{B,C}$ corresponds to the composite
\[
 G B^C\otimes GB\xrightarrow{\mu}
 G(B^C\otimes B)\xrightarrow{G(\ev)}GC.
\]

On $\Frob(\cC,G)^{\op}\times\Cart(\cC,G)$, the universal structure
arrows $\tau$ and $\kappa$ in the defining pullbacks are natural
transformations.
Contravariance in the first variable and covariance in the second make
$\tau^\kappa$ a natural transformation
$(GB)^{GC}\to B^C$ on this product.
Thus \eqref{eq:abstract-hom-structure} is a natural transformation $G(B^C)\to B^C$ on that product of lax equalizers.
The pullback universal property of $\Cart(\cC,G)$ gives the claimed functor.
\end{proof}

\section{Finite duality and opposite Frobenius structures}\label{sec:finite}
\subsection{Quasi-coherent complexes on the Zariski site}
Throughout this section $X$ is a Noetherian scheme.
We work only with the Zariski topology of $X$, and with the derived $\infty$-categories of ordinary rings from \cite[Definition 1.3.5.8; Theorem 7.1.2.13]{HA}.

\begin{definition}\label{def:dqc}
Let $\Aff_X$ be the category of affine open subschemes of $X$, with inclusions as morphisms.
The assignment $U\mapsto\cD(\Gamma(U,\cO_U))$ is a functor $\Aff_X^{\op}\to\Catbig_\infty$: an inclusion $\Spec R_f\to\Spec R$ is sent to derived extension of scalars $(-\otimes_R^{\mathrm L}R_f)$.
Set
\[
 \Dqc(X)=\lim_{U\in\Aff_X^{\op}}\cD\bigl(\Gamma(U,\cO_U)\bigr)
\]
in $\Catbig_\infty$.
For a not necessarily affine open $V\subset X$, put
\[
 \Dqc(V)=\lim_{U\in\Aff_V^{\op}}\cD(\Gamma(U,\cO_U)).
\]
Write also $\cD(X)$ for the stable derived $\infty$-category of the Grothendieck abelian category of all $\cO_X$-modules, in the sense of \Cref{lem:grothendieck-derived}.
The functor $\Fun(K,-):\Catbig_\infty\to\Catbig_\infty$ is right adjoint to cartesian product with $K$, by the exponential law for simplicial sets.
Right adjoints preserve limits, so
$$\Fun(K,\Dqc(X))\simeq\lim_{U\in\Aff_X^{\op}}\Fun(K,\cD(\Gamma(U,\cO_U)))$$
for every simplicial set $K$.
Taking $K=\Delta^0$ and $K=\Delta^1$, an object of $\Dqc(X)$ is a point of the limit of the spaces of objects, and a morphism is a point of the limit of the spaces of morphisms.
Unpacking the limit, this is a diagram: for each affine $U$ an object $M_U\in\cD(\Gamma(U,\cO_U))$, for each inclusion a specified equivalence $M_U\otimes^{\mathrm L}\Gamma(V,\cO_V)\simeq M_V$, and the usual higher data of a limit cone.
Write $\Dbc(X)\subset\Dqc(X)$ for the full subcategory of objects $M$ whose restriction to every affine open $U=\Spec R\subset X$ lies in $\mathcal D^b_{\mathrm{coh}}(R)\subset\cD(R)$, in the sense of \Cref{lem:bounded-serre} for $\cA=\Coh(R)$, the abelian category of coherent $R$-modules.
\end{definition}

\begin{lemma}\label{lem:projections-conservative}
Let $F:I\to\Catbig_\infty$ be a diagram and write $\cC=\lim F$.
The projection functors $\cC\to F(i)$ are jointly conservative: a morphism in $\cC$ is an equivalence if and only if each of its images in $F(i)$ is an equivalence.
\end{lemma}
\begin{proof}
Functors preserve equivalences, so it remains to reflect them.
Write $\pi_i:\cC\to F(i)$ for the projections.
The functor $(-)^{\Delta^1}$ is a right adjoint, as in \Cref{def:dqc}, hence preserves limits:
$(\cC)^{\Delta^1}\simeq\lim_i(F(i))^{\Delta^1}$.
For any $\infty$-category $\cE$, the mapping space $\Map_{\cE}(X,Y)$ is (represented by) the fibre
$\Hom_{\cE}(X,Y)=\{X\}\times_{\cE}\cE^{\Delta^1}\times_{\cE}\{Y\}$ of the source-and-target map
$(\cE)^{\Delta^1}\to\cE\times\cE$ over the point $(X,Y)$ \cite[Remark 1.2.2.5]{HTT}.
Applying this fibrewise to the limit cone for $\cC=\lim F$, and using that fibres commute with limits of spaces, one obtains
$\Map_{\cC}(X,Y)\simeq\lim_i\Map_{F(i)}(\pi_iX,\pi_iY)$.
A morphism $f$ in $\cC$ is an equivalence if and only if $\Map(Z,f)$ is an equivalence of spaces for every object $Z$, by the Yoneda embedding \cite[Proposition 5.1.3.1]{HTT}.
If each $\pi_i(f)$ is an equivalence, then each $\Map(\pi_iZ,\pi_i(f))$ is an equivalence of spaces, and a limit of equivalences of spaces is an equivalence.
Thus $f$ is an equivalence.
\end{proof}

\begin{lemma}\label{lem:dqc-affine}
For a commutative ring $R$ there is a canonical equivalence
\[
 \Dqc(\Spec R)\simeq\cD(R)\simeq\cMod_{HR}.
\]
Here $\cMod_{HR}$ is the stable $\infty$-category of modules over the Eilenberg--MacLane spectrum $HR$.
\end{lemma}
\begin{proof}
The second equivalence is \cite[Theorem 7.1.2.13]{HA}.
We reserve the roman notation $\Mod_R$ for the ordinary abelian category of $R$-modules.
We construct a functor $\ev:\Dqc(\Spec R)\to\cD(R)$ by evaluation at the identity open $\Spec R$, i.e., by the canonical projection from the limit to the corresponding term, and prove that $\ev$ is an equivalence of $\infty$-categories.

The identity open $\Spec R$ is terminal in $\Aff_{\Spec R}$.
The indexing category $I=\Aff_{\Spec R}^{\op}$ of the limit therefore has an initial object.
The inclusion of an initial object is left cofinal \cite[\href{https://kerodon.net/tag/03LQ}{Tag 03LQ}]{Kerodon}.
Left cofinal morphisms compute limits \cite[\href{https://kerodon.net/tag/02NS}{Tag 02NS}]{Kerodon}.
Thus $\ev$ is an equivalence in $\Catbig_\infty$.
In particular it induces equivalences of all mapping spaces, not merely a bijection on objects.

The inverse sends $M\in\cD(R)$ to the family $\bigl(M\otimes_R^{\mathrm L}\Gamma(U,\cO_U)\bigr)_U$.
This is the image of $M$ under the structure maps of the diagram, so it is a section of the limit.
\end{proof}

\begin{definition}\label{def:dualizing-dqc}
Let $X$ be a locally Noetherian scheme and let $\omega\in\Dqc(X)$.
We say that $\omega$ is \textit{dualizing} if, on every affine open $\Spec R\subset X$, the corresponding object of $\cD(R)$ is a dualizing complex over the Noetherian ring $R$ in the sense of \cite[\href{https://stacks.math.columbia.edu/tag/0A7B}{Tag 0A7B}]{SP}.
\end{definition}
Under \Cref{lem:dqc-affine} and \Cref{lem:bounded-serre}(2), this matches the usual local characterization of a dualizing complex in the triangulated category $D(\cO_X)$ \cite[\href{https://stacks.math.columbia.edu/tag/0A86}{Tag 0A86}; \href{https://stacks.math.columbia.edu/tag/0A87}{Tag 0A87}]{SP}.

\begin{lemma}\label{lem:dqc-coproduct}
If $X=X'\amalg X''$ is a finite coproduct of schemes, then $\Dqc(X)\simeq\Dqc(X')\times\Dqc(X'')$.
\end{lemma}
\begin{proof}
Every affine open of $X$ is a coproduct $U'\amalg U''$ of affine opens of the factors, possibly empty, and $\Aff_X\simeq\Aff_{X'}\times\Aff_{X''}$.
The abelian categories of modules satisfy $\Mod_{R\times S}\simeq\Mod_R\times\Mod_S$, hence $\mathrm{Ch}(\Mod_{R\times S})\simeq\mathrm{Ch}(\Mod_R)\times\mathrm{Ch}(\Mod_S)$ with componentwise quasi-isomorphisms.
The derived $\infty$-categories are the localizations of these chain categories at quasi-isomorphisms \cite[Definition 1.3.5.8; Proposition 1.3.5.15]{HA}, so $\cD(R\times S)\simeq\cD(R)\times\cD(S)$.
Limits of products are products of limits.
\end{proof}

\begin{lemma}\label{lem:affine-cech}
Let $R$ be a commutative ring and let $f_1,\dots,f_n\in R$ generate the unit ideal.
The \v Cech nerve of the standard open cover $\Spec R=\bigcup_i D(f_i)$ induces an equivalence
\[
 \cD(R)\simeq\lim_{[k]\in\Delta}\prod_{i_0,\dots,i_k}\cD(R_{f_{i_0}\cdots f_{i_k}}).
\]
\end{lemma}
\begin{proof}
Write $G:\cD(R)\to L$ for the comparison to the displayed limit, and $\mathrm{res}_j:\cD(R)\to\cD(R_{f_j})$ for localization at $f_j$.
Each restriction $\cD(R)\to\cD(R_{f_{i_0}\cdots f_{i_k}})$ is a morphism of $\Pr^{\mathrm L}$ \cite[Definition 5.5.3.1; Proposition 5.5.3.13]{HTT}.
The family $\{\mathrm{res}_j\}_j$ is jointly conservative: if $M\otimes_R^{\mathrm L}R_{f_i}\simeq0$ for all $i$, then $\operatorname{Supp}H^*(M)\subseteq\bigcap_i V(f_i)=\emptyset$, so $M\simeq0$ by nondegeneracy of the $t$-structure; equivalences are detected by cofibers.

After restriction to one $D(f_j)$, the cover splits \cite[Definition 4.7.2.2]{HA}.
Applying $\cD(-)$ and \Cref{lem:dqc-coproduct} (finite products by induction) yields a split augmented cosimplicial object $X_j^\bullet$, hence a limit diagram by \cite[\href{https://kerodon.net/tag/04SH}{Tag 04SH}]{Kerodon}, equivalently \cite[Remark 4.7.2.3]{HA} (via \cite[Lemma 6.1.3.16]{HTT}), using that $(-)^{\mathrm{op}}$ swaps limits and colimits \emph{in} one $\infty$-category (as opposed to limits \emph{of} $\infty$-categories in \Cref{prop:reversal}).
Write $G_j:\cD(R_{f_j})\to\lim X_j^\bullet$ for the comparison; thus $G_j$ is an equivalence.

Since $G$ is a morphism of $\Pr^{\mathrm L}$ between presentable $\infty$-categories, it admits a right adjoint $G^R$ \cite[Corollary 5.5.2.9; Proposition 5.5.3.13]{HTT}; write $\eta$ for the unit of $G\dashv G^R$.
Naturality of the \v Cech diagram in the base ring gives a commutative square of left adjoints relating $\mathrm{res}_j$ to the two \v Cech cones.
The mate of that square identifies $\mathrm{res}_j(\eta)$ with the unit of $G_j$.
Each $G_j$ is an equivalence, hence each $\mathrm{res}_j(\eta)$ is an equivalence, and joint conservativity of $\{\mathrm{res}_j\}$ yields that $\eta$ is an equivalence.
Thus $G$ is fully faithful: the triangle identities identify $\Map(M,N)\to\Map(GM,GN)$ with precomposition against $\eta$ \cite[Definition 1.2.10.1]{HTT}.

Write $\varepsilon$ for the counit of $G\dashv G^R$.
The same square of left adjoints determines a square of right adjoints $G^R$ and $G_j^R$.
Write $\varphi_j:L\to\lim X_j^\bullet$ for the functor obtained by restricting the \v Cech cone to $D(f_j)$.
The mate of that square of right adjoints identifies $\varphi_j(\varepsilon)$ with the counit of $G_j$.
Each $G_j$ is an equivalence, so each local counit is an equivalence.
The functors $\varphi_j$ are jointly conservative: the $0$-simplex projections $L\to\cD(R_{f_j})$ factor through $\varphi_j$ followed by a quasi-inverse of $G_j$, since the local nerve includes $D(f_j)$ as a $0$-simplex, and $\{L\to\cD(R_{f_j})\}_j$ is jointly conservative by \Cref{lem:projections-conservative}.
Equivalences in $L$ are therefore detected by the local comparisons $G_j$, not by identifying $\varepsilon$ with a counit after the $0$-simplex projections $L\to\prod_i\cD(R_{f_i})$.
Hence $\varepsilon$ is an equivalence.
An adjunction whose unit and counit are equivalences is an adjoint equivalence, so $G$ is an equivalence with inverse $G^R$.
\end{proof}

\begin{lemma}\label{lem:dqc-properties}
\begin{enumerate}
\item
The $\infty$-category $\Dqc(X)$ is stable, presentable, and closed symmetric monoidal.
Restriction to an open is symmetric monoidal and preserves colimits.
\item
The subcategory $\Dbc(X)$ is stable.
\item
Tensor product preserves colimits separately, so internal Hom is a functor
\[
 \RHom:\Dqc(X)^{\op}\times\Dqc(X)\longrightarrow\Dqc(X),
\]
obtained by adjunction \cite[Corollary 5.5.2.9]{HTT}.
\end{enumerate}
\end{lemma}
\begin{proof}
Each $\cD(R)$ is stable and presentable \cite[Proposition 1.3.5.9, Proposition 1.3.5.21]{HA}, and $\cD(R)\in\CAlg(\Pr^{\mathrm L})$ by \cite[Theorem 7.1.2.13; Definition 4.1.1.15]{HA} (tensor preserves colimits separately) together with the symmetric monoidal structure on $\Pr^{\mathrm L}$ of \cite[Proposition 4.8.1.15]{HA}.
Affine restriction is a morphism of $\CAlg(\Pr^{\mathrm L})$, so
\[
 U\longmapsto\cD\bigl(\Gamma(U,\cO_U)\bigr)
\]
is a diagram $\Aff_X^{\op}\to\CAlg(\Pr^{\mathrm L})\subset\Pr^{\mathrm L}\subset\Catbig_\infty$.
Limits in $\CAlg(\Pr^{\mathrm L})$ are created in $\Pr^{\mathrm L}$ \cite[Corollary 3.2.2.5]{HA}, and limits in $\Pr^{\mathrm L}$ are created in $\Catbig_\infty$ \cite[Proposition 5.5.3.13]{HTT}, so
\[
 \Dqc(X)=\lim_{U\in\Aff_X^{\op}}\cD\bigl(\Gamma(U,\cO_U)\bigr)\in\CAlg(\Pr^{\mathrm L}).
\]
In particular $\Dqc(X)$ is presentable and closed symmetric monoidal, restriction is strong symmetric monoidal and preserves colimits, and internal Hom exists as in \textup{(3)}.

Transitions are $t$-exact, so limit projections create finite limits and colimits affine-locally.
Stability of $\Dqc(X)$ follows from stability of each $\cD(R)$ and \cite[Corollary 1.4.2.27]{HA} affine-locally.

For (2): equivalences and fibres in $\Dqc(X)$ are detected on affines by \Cref{lem:projections-conservative}.
On $U=\Spec R$ Noetherian, $\mathcal D^b_{\mathrm{coh}}(R)\subset\cD(R)$ is \Cref{lem:bounded-serre} for $\cA=\Coh(R)\subset\Mod_R$.
If $M,N\in\Dbc(X)$, then on each affine open the restrictions lie in $\mathcal D^b_{\mathrm{coh}}(R)$, so the fibre of $M\to N$ restricts to an object of $\mathcal D^b_{\mathrm{coh}}(R)$ by \Cref{lem:bounded-serre}.
Vanishing of cohomology on an affine is the sheaf axiom, and is local on a cover.
Thus the fibre lies in $\Dbc(X)$.
Shifts likewise preserve $\Dbc(X)$.
Stability of $\Dbc(X)$ is then \cite[Lemma 1.1.3.3]{HA}.
\end{proof}

\begin{lemma}\label{lem:dqc-zariski}
Let $\{V_i\}_{i=1}^r$ be a finite Zariski open cover of $X$.
Write $V_n=\coprod_{i_0,\dots,i_n}(V_{i_0}\cap\cdots\cap V_{i_n})$ for the $n$-simplices of its \v Cech nerve.
Restriction induces an equivalence
\[
 \Dqc(X)\simeq\lim_{[n]\in\Delta}\Dqc(V_n)
\]
in $\Catbig_\infty$.
\end{lemma}
\begin{proof}
We first treat $U=\Spec R$ and a finite cover $\{U_j\}$ of $U$ by affine opens.
Intersections of affine opens in an affine scheme are affine, so every \v Cech stage is affine.
Refine $\{U_j\}$ by a finite standard cover $\{D(f_k)\}$ with each $D(f_k)$ contained in some $U_{j(k)}$.
Write $G$ for the comparison $\Dqc(U)\to\lim_{[n]}\Dqc(U_n)$.
The standard affine \v Cech comparison is an equivalence by \Cref{lem:affine-cech,lem:dqc-affine}.
After restriction to one $D(f_k)$, the coproduct of the opens $U_j\cap D(f_k)$ maps to $D(f_k)$ by a split epimorphism, so the pulled-back nerve is split \cite[Definition 4.7.2.2]{HA}.
Applying $\Dqc$ and \Cref{lem:dqc-coproduct} yields a split augmented cosimplicial object, hence a limit diagram by the same opposition of \cite[\href{https://kerodon.net/tag/04SH}{Tag 04SH}]{Kerodon}, equivalently \cite[Remark 4.7.2.3]{HA}, as in \Cref{lem:affine-cech}.
Thus the local comparison functors on each $D(f_k)$ are equivalences.
As in \Cref{lem:affine-cech}, $G$ is a morphism of $\Pr^{\mathrm L}$, so $G$ admits a right adjoint $G^R$ \cite[Corollary 5.5.2.9]{HTT}; write $\eta$ for the unit and $\varepsilon$ for the counit.
The commutative square of left adjoints relating $\mathrm{res}_{D(f_k)}$ to the two \v Cech cones is a morphism of adjunctions, as in \Cref{lem:affine-cech}, so $\mathrm{res}_{D(f_k)}(\eta)$ is identified with the unit of the local split comparison, and $\eta$ is an equivalence by joint conservativity of $\{\mathrm{res}_{D(f_k)}\}$.
The same square determines a square of right adjoints.
Write $\varphi_k$ for restriction of the \v Cech cone of $L=\lim_{[n]}\Dqc(U_n)$ to $D(f_k)$.
The mate of that square of right adjoints identifies $\varphi_k(\varepsilon)$ with the counit of the local split comparison, which is an equivalence.
The functors $\varphi_k$ are jointly conservative: they detect equivalences in $L$ by the local comparisons, which are the local $G_j$ of \Cref{lem:affine-cech}, together with joint conservativity of $\{\mathrm{res}_{D(f_k)}\}$.
Hence both the unit and the counit are equivalences, so $G$ is an equivalence.

Now let $\{W_j\}$ be an arbitrary finite Zariski cover of the same affine $U$.
Cover each $W_j$ by finitely many affine opens of $U$, obtaining an affine cover $\{U_\alpha\}$ which refines $\{W_j\}$.
Write $L_W=\lim_{[n]}\Dqc(W_n)$ and $L_U=\lim_{[n]}\Dqc(U_{\alpha,n})$.
The comparison $\Dqc(U)\to L_U$ is an equivalence by the previous paragraph.
The refinement of nerves gives $L_W\to L_U$, and the composite $\Dqc(U)\to L_W\to L_U$ is the affine comparison, hence an equivalence, so $\Dqc(U)\to L_W$ is fully faithful.
It remains to see that $L_W\to L_U$ is conservative.
If $E\in L_W$ maps to zero in $L_U$, then for each $j$ the component $E_j\in\Dqc(W_j)$ vanishes on each affine $U_\alpha\subset W_j$.
For an arbitrary affine $T\subset W_j$, the intersections $T\cap U_\alpha$ are affine opens covering $T$, and $E_j|_T$ vanishes on that affine cover, so $E_j|_T\simeq0$ by the affine-cover case already proved.
Thus $E_j\simeq0$ by the defining limit cone of $\Dqc(W_j)$, and $E\simeq0$.
Therefore $\Dqc(U)\to L_W$ is an equivalence.

Now let $X$ be arbitrary.
By definition and \Cref{lem:dqc-affine}, $\Dqc(X)\simeq\lim_{U\in\Aff_X^{\op}}\Dqc(U)$.
For each affine $U\subset X$, the pulled-back cover $\{U\cap V_i\}$ of $U$ satisfies $\Dqc(U)\simeq\lim_{[n]}\Dqc(U\cap V_n)$ by the previous paragraph.
Hence
\[
 \Dqc(X)\simeq\lim_{U\in\Aff_X^{\op}}\lim_{[n]\in\Delta}\Dqc(U\cap V_n).
\]
Both iterated limits are the limit of $(U,[n])\mapsto\Dqc(U\cap V_n)$ over $\Aff_X^{\op}\times\Delta$, so they may be taken in either order.

It remains to identify $\lim_{U\in\Aff_X^{\op}}\Dqc(U\cap V)\simeq\Dqc(V)$ for each open $V\subset X$, and then apply \Cref{lem:dqc-coproduct} to the coproducts $V_n$.
By definition, $\Dqc(U\cap V)\simeq\lim \Dqc(W)$ over $W\in\Aff_{U\cap V}^{\op}$.
Those $W$ are exactly the objects of $\Aff_V$ contained in $U$.
The double limit is therefore the limit of $\Dqc(W)$ over pairs $(U,W)$ with $W\subset U\cap V$.
For fixed $W_0\in\Aff_V$, the inner diagram in the variable $U$ (ranging over affine opens of $X$ containing $W_0$, as a full subcategory of $\Aff_X^{\op}$) is the constant diagram with value $\Dqc(W_0)$.
In $\Aff_X^{\op}$ morphisms run from larger opens to smaller ones, so $U=W_0$ is a terminal object of that slice.
An $\infty$-category with a terminal object is weakly contractible \cite[\href{https://kerodon.net/tag/02P2}{Tag 02P2}]{Kerodon}.
The constant cone on $\Dqc(W_0)$ carries every edge to an identity, hence is a limit diagram \cite[\href{https://kerodon.net/tag/02XX}{Tag 02XX}]{Kerodon}.
The remaining limit over $W\in\Aff_V^{\op}$ is $\Dqc(V)$.
\end{proof}

\begin{lemma}\label{lem:finite-push}
Let $f:X\to Y$ be a finite morphism of Noetherian schemes.
There is a colimit-preserving pushforward $f_*:\Dqc(X)\to\Dqc(Y)$, with left adjoint $f^*$ and right adjoint $f^!$, such that on affines,
\begin{itemize}
\item
$f_*$ is restriction of scalars along the finite ring map;
\item
$f^*$ is derived extension of scalars;
\item
$f^!$ is $\mathrm{RHom}_R(S,-)$.
\end{itemize}
\end{lemma}
\begin{proof}
A finite morphism is affine \cite[\href{https://stacks.math.columbia.edu/tag/01WN}{Tag 01WN}]{SP}, so $V\mapsto f^{-1}(V)$ is an equivalence $\Aff_Y\simeq\Aff_X$, and $V=\Spec R$ gives $f^{-1}(V)=\Spec S$ with $R\to S$ finite.

On affines, $\mathrm{res}_{R\to S}:\cD(S)\to\cD(R)$ is $\cMod_{HS}\to\cMod_{HR}$ \cite[Theorem 7.1.2.13]{HA}, and lies in $\Pr^{\mathrm L}$ (in fact preserves limits and colimits).
Compatibility with localization $R\to R_g$ (via $S\otimes_R R_g\simeq S_g$) yields
\[
 \Phi\in\Fun(\Aff_Y^{\op},\Catbig_\infty),\qquad
 \Phi_V=\mathrm{res}_{R_V\to S_V}:\cD(S_V)\to\cD(R_V).
\]
Set $f_*=\lim\Phi:\Dqc(X)\to\Dqc(Y)$ \cite[Corollary 4.2.4.8; Remark 3.0.0.5]{HTT}; then $f_*$ is a morphism of $\Pr^{\mathrm L}$ by \cite[Proposition 5.5.3.13]{HTT}.

Likewise $(-\otimes_R^{\mathrm L}S)$ induces $f^*:\Dqc(Y)\to\Dqc(X)$, and the affine adjunctions $(-\otimes_R^{\mathrm L}S)\dashv\mathrm{res}_{R\to S}$ are compatible with localization, so their units and counits pass to the limits.
Mapping spaces in a limit are limits of mapping spaces (\Cref{lem:projections-conservative}), hence $f^*\dashv f_*$ on $\Dqc$.

Presentability and the fact that $f_*$ is a morphism of $\Pr^{\mathrm L}$ give a right adjoint $f^!$ by \cite[Corollary 5.5.2.9]{HTT}.
On $V=\Spec R$ one has $f^!|_V\simeq\mathrm{RHom}_R(S,-)$ by closed monoidality of $\cD(R)\simeq\cMod_{HR}$, and uniqueness of right adjoints \cite[Remark 5.2.2.2]{HTT}.
\end{proof}

\begin{lemma}\label{lem:projection}
In the situation of \Cref{lem:finite-push}, there is a natural equivalence
\[
 f_*(f^*A\otimes M)\simeq A\otimes f_*M
\]
in $\Dqc(Y)$, for $A\in\Dqc(Y)$ and $M\in\Dqc(X)$.
\end{lemma}
\begin{proof}
Finite morphisms are affine, so $f_*$ and $f^*$ agree with $Rf_*$ and $Lf^*$ on $\mathrm D_{\mathrm{qc}}$.
On that triangulated category the displayed formula is \cite[\href{https://stacks.math.columbia.edu/tag/08EU}{Tag 08EU}]{SP}.
The comparison morphism already lies in $\Fun(\Dqc(Y)\times\Dqc(X),\Dqc(Y))$ by construction of $f_*$, $f^*$ and $\otimes$; it is an equivalence because affine restrictions are jointly conservative (\Cref{lem:projections-conservative}) and the triangulated statement applies on each affine via \Cref{lem:dqc-affine}.
\end{proof}

\begin{remark}[Relation with abstract 6FF]
\label{rem:no-6ff-import}
\Cref{lem:finite-push} gives a right adjoint $f^!$ for finite $f$, and \Cref{lem:projection} gives the projection formula $f_*(f^*A\otimes M)\simeq A\otimes f_*M$.
Scholze's construction principle takes adjoints, base change and the projection formula as \emph{input} \cite[Theorem 4.6]{Scholze-6ff}, after Mann \cite[Proposition A.5.10]{Mann}; it does not prove those facts for Zariski $\Dqc$.
The formalisms of Lecture VIII are $\Pro(\Perf)$, $\Pro(\Coh)$ and $\Ind(\Coh)$, not Zariski $\Dqc$ of ordinary modules \cite[Theorem 8.21; \S8.5--8.7]{Scholze-6ff}.
None of these references yields \Cref{lem:finite-push,lem:projection} for Noetherian schemes, nor the lax equalizers $\Frob(\Dbc,F_*)$ and $\Cart(\Dbc,F_*)$.
\end{remark}

\begin{lemma}\label{lem:grothendieck}
Let $\omega\in\Dqc(X)$ be dualizing in the sense of \Cref{def:dualizing-dqc}.
The functor
\begin{equation}\label{eq:underlying-dual}
 D_\omega=\RHom(-,\omega):\Dbc(X)^{\op}\simeq\Dbc(X)
\end{equation}
is an exact equivalence.
The unit of tensor-Hom is a natural equivalence $\id\to D_\omega\circ D_\omega^{\op}$ in $\Fun(\Dbc(X),\Dbc(X))$, and likewise $D_\omega^{\op}\circ D_\omega\simeq\id$ on $\Dbc(X)^{\op}$.
\end{lemma}
\begin{proof}
Internal Hom of \Cref{lem:dqc-properties} is exact on $\Dqc(X)$ (right adjoint in the second variable).
The unit $\eta:\id\to D_\omega\circ D_\omega^{\op}$ of tensor-Hom is a morphism in $\Fun(\Dqc(X),\Dqc(X))$.
On each affine $U=\Spec R$, \Cref{lem:dqc-affine} identifies $D_\omega$ with $\mathrm{RHom}_R(-,\omega|_U)$; by \cite[\href{https://stacks.math.columbia.edu/tag/0A7C}{Tag 0A7C}]{SP} this preserves $\mathcal D^b_{\mathrm{coh}}(R)$ and $\eta|_U$ is an isomorphism in the homotopy category, hence an equivalence in $\cD(R)$.
Affine restrictions are jointly conservative (\Cref{lem:projections-conservative}), so $D_\omega$ preserves $\Dbc(X)$ and $\eta$ is an equivalence on $\Dbc(X)$.
The opposite transformation on $\Dbc(X)^{\op}$ is the same unit, hence likewise an equivalence.
\end{proof}

\begin{lemma}\label{lem:open-rhom}
Let $f:X\to Y$ be finite, let $j_Y:V\to Y$ be a quasi-compact open immersion of Noetherian schemes, and write $X_V=X\times_Y V$, with projections $j_X:X_V\to X$ and $f_V:X_V\to V$.
The base-change transformation
\[
 j_X^*f^!\to f_V^!j_Y^*
\]
is an equivalence on objects of $\Dqc(Y)$ that are bounded below.
In particular it is an equivalence on $\Dbc(Y)$ and on dualizing complexes, so $f^!$ commutes with open restriction on those objects.
\end{lemma}
\begin{proof}
Write
\[
 \beta:j_X^*f^!\to f_V^!j_Y^*
\]
for the mate of the pushforward base-change equivalence
\[
 f_{V,*}j_X^*\simeq j_Y^*f_*
\]
under the adjunctions $f_*\dashv f^!$ and $f_{V,*}\dashv f_V^!$, using \Cref{lem:finite-push} (on affines of $V$ both sides are restriction of scalars along the same finite ring map).
Thus $\beta$ is a morphism in $\Fun(\Dqc(Y),\Dqc(X_V))$.
It is an equivalence if and only if its restrictions to affine opens of $X_V$ are equivalences (\Cref{lem:projections-conservative}; \cite[\href{https://kerodon.net/tag/01DK}{Tag 01DK}]{Kerodon}).
On such an affine, \Cref{lem:finite-push,lem:dqc-affine} identify $\beta$ with the map of \cite[\href{https://stacks.math.columbia.edu/tag/0A6A}{Tag 0A6A(3)}]{SP}; the finite module is pseudo-coherent \cite[\href{https://stacks.math.columbia.edu/tag/066E}{Tag 066E(3)}]{SP}.
For $K$ bounded below that map is an isomorphism in the homotopy category, hence an equivalence.
Objects of $\Dbc(Y)$ and dualizing complexes are bounded below in this sense (\Cref{def:dualizing-dqc}), so $\beta$ is an equivalence on those objects.
\end{proof}

In the rest of this section, $\RHom$ denotes the internal Hom of \Cref{lem:dqc-properties}.
For a bounded coherent first input and a dualizing second input, restriction to an affine open identifies this internal Hom with module $\mathrm{RHom}_R$, by the affine equivalence of \Cref{lem:dqc-affine} and the monoidal limit construction in \Cref{lem:dqc-properties}.
(Comparison with derived sheaf Hom in the triangulated category $D(\cO_U)$ is the usual one of \cite[\href{https://stacks.math.columbia.edu/tag/0A88}{Tag 0A88}]{SP}; it is a statement in $D(\cO_U)$, not a morphism in $\Dqc$.)
The same internal Hom restricts to the equivalence $D_\omega$ of \Cref{lem:grothendieck}.
For a finite morphism $f:X\to Y$, the functor $f^!$ is the right adjoint of \Cref{lem:finite-push}; on affines it is $\mathrm{RHom}_R(S,-)$.

\begin{proposition}\label{prop:finite-duality}
Let $f:X\to Y$ be finite between Noetherian schemes.
For $M\in\Dqc(X)$ and $K\in\Dqc(Y)$ there is a natural equivalence
\begin{equation}\label{eq:finite-duality}
 f_*\RHom_X(M,f^!K)\simeq\RHom_Y(f_*M,K).
\end{equation}
If $X=Y$, $f=F$, and $u:\omega\simeq F^!\omega$ is a chosen equivalence, this gives
\begin{equation}\label{eq:gamma}
 \gamma:F_*D_\omega\xrightarrow{\ \simeq\ }D_\omega F_*^{\op}
\end{equation}
as an equivalence in $\Fun(\Dbc(X)^{\op},\Dbc(X))$.
\end{proposition}
\begin{proof}
On the triangulated category $\mathrm D_{\mathrm{qc}}$, the comparison morphism of \cite[\href{https://stacks.math.columbia.edu/tag/0B6H}{Tag 0B6H}]{SP} is an isomorphism after applying $DQ_Y$, by \cite[\href{https://stacks.math.columbia.edu/tag/0A9Q}{Tag 0A9Q}]{SP} (affine case: restriction of scalars along a finite ring map).
The comparison morphism lies in $\Fun(\Dqc(X)^{\op}\times\Dqc(Y),\Dqc(Y))$ by construction of $f_*$, $f^!$ and $\RHom$; it is an equivalence by joint conservativity of affine restrictions (\Cref{lem:projections-conservative}) together with \Cref{lem:dqc-affine,lem:finite-push}.

For \eqref{eq:gamma}, compose with $F_*\RHom(M,u)$ and restrict the source to $\Dbc(X)^{\op}$.
Since \eqref{eq:finite-duality} and $u$ are equivalences, the resulting $\gamma$ is an equivalence in $\Fun(\Dbc(X)^{\op},\Dqc(X))$.
On affines $F_*$ preserves $\Dbc$ \cite[\href{https://stacks.math.columbia.edu/tag/00GJ}{Tag 00GJ}]{SP}.
On $X$, locality of coherence \cite[\href{https://stacks.math.columbia.edu/tag/01XZ}{Tag 01XZ}]{SP} and locality of cohomological vanishing on a finite cover give the same conclusion.
With \Cref{lem:grothendieck}, both sides of $\gamma$ land in $\Dbc(X)$, so $\gamma$ lies in $\Fun(\Dbc(X)^{\op},\Dbc(X))$.
\end{proof}

\begin{remark}[$F_*^{\op}$ in \eqref{eq:gamma}]\label{rem:Fop-gamma}
The duality $D_\omega$ is a functor $\Dbc(X)^{\op}\to\Dbc(X)$.
The opposite functor $F_*^{\op}:\Dbc(X)^{\op}\to\Dbc(X)^{\op}$ is $F_*$ on objects and $F_*(f)^{\op}$ on morphisms.
The composite $D_\omega\circ F_*^{\op}$ therefore sends $M$ to $D_\omega(F_*M)$, while $F_*\circ D_\omega$ sends $M$ to $F_*D_\omega(M)$.
The equivalence $\gamma$ identifies these two functors $\Dbc(X)^{\op}\to\Dbc(X)$.
This is not a statement about the pullback $F^*$.
The pullback $F^*$ is a covariant functor $\Dqc(X)\to\Dqc(X)$, left adjoint to $F_*$.
Even if $F^*$ preserves $\Dbc(X)$, so that $(F^*)^{\op}$ acts on $\Dbc(X)^{\op}$, the composites $D_\omega\circ (F^*)^{\op}$ and $F^*\circ D_\omega$ are different functors from both sides of $\gamma$.
The adjunction $F^*\dashv F_*$ yields a covariant equivalence $\Cart(\cC,F^*)\simeq\Frob(\cC,F_*)$ \cite[Proposition 2.9; Corollary 2.10]{MWC}, which does not reverse arrows.
The input used here is the identification $u:\omega\simeq F^!\omega$ together with \eqref{eq:finite-duality}, i.e., Grothendieck duality for the finite morphism $F$, not the left adjoint $F^*$.
\end{remark}

\begin{definition}[{cf.\ \cite[Definition 4.2.1]{Bau}}]\label{def:unit-dualizing}
Let $X$ be a Noetherian $F$-finite scheme.
A complex $\omega$ in $D^+(\Fil\Coh^{C,\mathrm{unit}}_X)$ is a \emph{unit dualizing complex} if $R\iota(\omega)$ is a dualizing complex on $X$, where $\iota$ denotes the inclusion $\Fil\Coh^{C,\mathrm{unit}}_X\to\operatorname{Mod}(\cO_X)$ of \cite[Definition 4.2.1]{Bau}.
The functor $\iota$ is not exact in general (\Cref{rem:resolution}), so one uses $R\iota(\omega)$.
Here $\Fil\Coh^{C,\mathrm{unit}}_X$ denotes the abelian category of unit Cartier modules whose underlying $\cO_X$ is a filtered colimit of coherent $\cO_X$-modules.
This is the category denoted $\mathrm{IndCoh}^{C,\mathrm{unit}}_X$ in \cite[Definition 3.4.5]{Bau}; we write $\Fil\Coh^{C,\mathrm{unit}}_X$ to reserve $\Ind$ for the $\infty$-category of ind-objects to come later.
\end{definition}

\begin{remark}[Unit dualizing data]\label{rem:unit-choice}
By \cite[Remark 4.2.2]{Bau}, the datum of a unit dualizing complex is equivalent to a dualizing complex in the ordinary derived category together with an isomorphism to its $F^!$-transform there.
Via the affine identifications of \Cref{lem:dqc-affine,lem:finite-push} and joint conservativity of affine restrictions (\Cref{lem:projections-conservative}), those data reconstruct a pair $(\omega,u)$ with $\omega\in\Dqc(X)$ dualizing in the sense of \Cref{def:dualizing-dqc} and an equivalence $u:\omega\simeq F^!\omega$ in $\Dqc(X)$, as needed for \Cref{prop:finite-duality,thm:intrinsic}.
We do not treat Baudin's $1$-categorical morphism as a literal morphism in $\Dqc(X)$; below we take $(\omega,u)$ as given $\infty$-categorical input.
The construction of $\gamma$ in \Cref{prop:finite-duality} is natural in the dualizing complex $K$ and, after postcomposition with $F_*\RHom(M,u)$, natural in equivalences of unit pairs $(\omega,u)$.
Consequently $D_\omega^{\mathrm{LEq}}$ of \Cref{thm:intrinsic} varies naturally in such equivalences.
\end{remark}

\begin{theorem}\label{thm:intrinsic}
Let $X$ be Noetherian and $F$-finite, with a dualizing complex $\omega$ and a chosen equivalence $u:\omega\simeq F^!\omega$.
Then \eqref{eq:underlying-dual} and \eqref{eq:gamma} induce an exact equivalence
\[
 D^{\mathrm{LEq}}_\omega:
 \Frob(\Dbc(X),F_*)^{\op}\simeq\Cart(\Dbc(X),F_*).
\]
On objects one has
\[
 D_\omega^{\mathrm{LEq}}(M,\tau)=\bigl(D_\omega M,\ D_\omega(\tau)\circ\gamma_M\bigr).
\]
A quasi-inverse sends $(N,\kappa:F_*N\to N)$ to $\bigl(D_\omega N,\ \sigma\bigr)$, where the Frobenius map of the image is
\[
 \sigma:
 D_\omega N
 \xrightarrow{D_\omega(\kappa)}
 D_\omega(F_*N)
 \xrightarrow{\gamma_N^{-1}}
 F_*D_\omega N.
\]
\end{theorem}
\begin{proof}
By \Cref{prop:finite-duality}, $\gamma$ is an equivalence in $\Fun(\Dbc(X)^{\op},\Dbc(X))$.
Apply \Cref{prop:reversal} to the pair $(\id,F_*)$ with $D=D_\omega$.
The required natural equivalences in $\Fun(\Dbc(X)^{\op},\Dbc(X))$ are the identity on $D_\omega$ and $\gamma^{-1}:D_\omega\circ F_*^{\op}\simeq F_*\circ D_\omega$.
A Frobenius structure $\tau:M\to F_*M$ is sent to the Cartier structure $D_\omega(\tau)\circ\gamma_M:F_*D_\omega M\to D_\omega M$.
The categories are stable by \cite[Proposition 2.6(e)]{MWC}, since $F_*$ is exact; an equivalence between stable $\infty$-categories is exact.
\end{proof}

\begin{theorem}\label{thm:bidual-leq}
Let $X$ be Noetherian and $F$-finite, with a dualizing complex $\omega$ and a chosen equivalence $u:\omega\simeq F^!\omega$.
Write $D_\omega^{\mathrm{LEq}}$ for the equivalence of \Cref{thm:intrinsic}, and write $\gamma$ for the natural equivalence of \Cref{prop:finite-duality}.
Let
\begin{equation}\label{eq:bidual-eta}
 \eta:\mathrm{id}_{\Dbc(X)}\xrightarrow{\ \simeq\ }D_\omega\circ D_\omega^{\op}
\end{equation}
be the unit of tensor-Hom in $\Fun(\Dbc(X),\Dbc(X))$ from \Cref{lem:grothendieck}, and let $\gamma^\vee$ be the mate of $\gamma^{-1}$ under $\eta$.
Write
\[
 R:\Cart(\Dbc(X),F_*)^{\op}\xrightarrow{\ \simeq\ }\Frob(\Dbc(X),F_*)
\]
for the equivalence obtained by applying \Cref{prop:reversal} to $D=D_\omega^{\op}$ with natural equivalences given by $\gamma^\vee$.
Then:
\begin{enumerate}
\item
There is a canonical equivalence
\[
 (D_\omega^{\mathrm{LEq}})^{\op}\circ D_\omega^{\mathrm{LEq}}
 \simeq
 \mathrm{id}_{\Frob(\Dbc(X),F_*)}
\]
in $\Fun\bigl(\Frob(\Dbc(X),F_*),\Frob(\Dbc(X),F_*)\bigr)$ whose underlying arrow is $\eta$.
\item
The analogous equivalence holds on $\Cart(\Dbc(X),F_*)$:
\[
 D_\omega^{\mathrm{LEq}}\circ (D_\omega^{\mathrm{LEq}})^{\op}
 \simeq
 \mathrm{id}_{\Cart(\Dbc(X),F_*)}
\]
in $\Fun\bigl(\Cart(\Dbc(X),F_*),\Cart(\Dbc(X),F_*)\bigr)$.
\item
The functor $R^{\op}:\Cart(\Dbc(X),F_*)\to\Frob(\Dbc(X),F_*)^{\op}$ realizes the objectwise quasi-inverse of \Cref{thm:intrinsic} in the functor $\infty$-category: the composites of $D_\omega^{\mathrm{LEq}}$ with $R^{\op}$ are the equivalences of \textup{(1)} and \textup{(2)}.
\end{enumerate}
\end{theorem}
\begin{proof}
By \Cref{lem:grothendieck}, $\eta$ already lies in $\Fun(\Dbc(X),\Dbc(X))$, and likewise $D_\omega^{\op}\circ D_\omega\simeq\mathrm{id}$ on $\Dbc(X)^{\op}$.
Conjugate $\gamma^{-1}:D_\omega\circ F_*^{\op}\simeq F_*\circ D_\omega$ by $\eta$ to obtain $\gamma^\vee$, a natural equivalence for $D_\omega^{\op}$ against $F_*$ of the variance required by a second application of \Cref{prop:reversal}.
Yoneda uniqueness as in the proof of \Cref{prop:finite-duality} determines $\gamma^\vee$ up to a contractible space of choices: for every test object $A$, the induced maps of mapping spaces are those of $\gamma^{-1}$ conjugated by $\eta$.
Triangle identities for mates identify the composite natural equivalence on structure maps with the identity natural equivalence for $(\id,F_*)$.

Apply \Cref{prop:reversal} to $D=D_\omega$ with natural equivalences $(\mathrm{id},\gamma^{-1})$ as in the proof of \Cref{thm:intrinsic}, obtaining $D_\omega^{\mathrm{LEq}}:\Frob(\Dbc(X),F_*)^{\op}\simeq\Cart(\Dbc(X),F_*)$.
Apply \Cref{prop:reversal} again to $D=D_\omega^{\op}:\Dbc(X)\simeq\Dbc(X)^{\op}$ with natural equivalences given by $\gamma^\vee$, obtaining $\Cart(\Dbc(X),F_*)^{\op}\simeq\Frob(\Dbc(X),F_*)$.
The composite endofunctor of $\Frob(\Dbc(X),F_*)$ has underlying functor $D_\omega\circ D_\omega^{\op}\simeq\mathrm{id}$ via $\eta$, and structure-path composite equivalent to the identity by the mate triangles above.
Hence the composite is the identity morphism in the functor $\infty$-category, which is \textup{(1)}.
Item \textup{(2)} is the symmetric composite with the roles of $\Frob$ and $\Cart$ exchanged.
The second application of \Cref{prop:reversal} above is exactly the functor $R$ of the statement; item \textup{(3)} records that $R^{\op}$ realizes the quasi-inverse of \Cref{thm:intrinsic}, with the two composites identified with \textup{(1)} and \textup{(2)}.
\end{proof}

\begin{proposition}\label{prop:trace-formula}
Let $X$ be Noetherian and $F$-finite, with a dualizing complex $\omega$ and a chosen equivalence $u:\omega\simeq F^!\omega$.
Let $\kappa_\omega:F_*\omega\to\omega$ be the adjoint of $u$, and let $(M,\tau)$ be an object of $\Frob(\Dbc(X),F_*)$.
Write
\[
 \kappa_M:
 F_*D_\omega M
 \xrightarrow{\gamma_M}
 D_\omega(F_*M)
 \xrightarrow{D_\omega(\tau)}
 D_\omega M
\]
for the Cartier structure map of $D_\omega^{\mathrm{LEq}}(M,\tau)$ in \Cref{thm:intrinsic}.
Then:
\begin{enumerate}
\item
As a morphism in $\Fun(\Dbc(X)^{\op},\Dbc(X))$, the map $\kappa_M$ coincides with the composite
\[
 F_*\RHom(M,\omega)
 \xrightarrow{\ \simeq\ }
 \RHom(F_*M,\omega)
 \xrightarrow{\RHom(\tau,\omega)}
 \RHom(M,\omega),
\]
where the first arrow is the finite-duality identification \eqref{eq:finite-duality} at $K=\omega$, composed with $F_*\RHom(M,u)$, and is adjoint to pairing against $\kappa_\omega$.
\item
The map $\kappa_M$ of \textup{(1)} has the following local description: on each open $U\subset X$, it restricts to the composite of sheaf morphisms
\[
 F_*\shHom(M,\omega)\longrightarrow\shHom(F_*M,\omega)\longrightarrow\shHom(M,\omega)
\]
sending a local section $s$ of $F_*\shHom(M,\omega)$ to
\begin{equation}\label{eq:trace-formula}
 \kappa_\omega\circ F_*s\circ\tau.
\end{equation}
\item
If $I$ is an injective complex representing $\omega$ and satisfying the unit condition (so that $u$ is represented by a map of complexes $I\to F^!I$), then the same formula in each degree gives a model of the map $\kappa_M$ of \textup{(1)}:
\begin{equation}\label{eq:trace-formula-complex}
 F_*\Hom^\bullet(M,I)\longrightarrow\Hom^\bullet(F_*M,I),\qquad
 s\longmapsto\kappa_I\circ F_*s\circ\tau.
\end{equation}
Here $s$ denotes a degree-zero chain map of sheaves on $U$, and the formula is understood in the homotopy category of complexes before passage to $\Dbc(X)$.
\end{enumerate}
\end{proposition}
\begin{proof}
\begin{enumerate}
\item
By the Yoneda argument in the proof of \Cref{prop:finite-duality}, the morphism $\gamma$ in $\Fun(\Dbc(X)^{\op},\Dbc(X))$ is uniquely characterized, up to a contractible space of choices, by the property that for every test object $A$ the induced map of mapping spaces
\[
 \Map\bigl(A,F_*D_\omega(M)\bigr)
 \longrightarrow
 \Map\bigl(A,D_\omega(F_*M)\bigr)
\]
is the comparison of \eqref{eq:finite-duality} evaluated at $K=\omega$, composed with $F_*\RHom(M,u)$ when $f=F$.
The last two of those equivalences are $F_*\dashv F^!$ and tensor-Hom on $X$.
Their composite, evaluated at $A=\cO$, is adjoint to evaluation at $1$, i.e., to the counit $F_*F^!\omega\to\omega$.
The chosen equivalence $u$ identifies that counit with $\kappa_\omega$.
Composing with $D_\omega(\tau)$ therefore identifies $\kappa_M$ with the composite in \textup{(1)}.
\item
Restricting the five mapping-space equivalences to affine opens and to discrete $\cO$-modules, the resulting $\pi_0$ statement is the sheaf isomorphism
\[
 F_*\shHom(M,N)\simeq\shHom(F_*M,N)
\]
obtained from $N\simeq F^\flat N$ and the affine formula of \Cref{lem:finite-push}, with formula $h\mapsto\kappa_NF_*h$ \cite[Lemma 3.4.1, first item and formula (3.4.1.c)]{Bau}.
Composition with $\tau$ gives the composite of sheaf morphisms in \textup{(2)}, with formula \eqref{eq:trace-formula}, which is the pairing of \cite[Construction 4.1.1]{Bau}.
This is the restriction of the map $\kappa_M$ of \textup{(1)} to opens and to sheaves.
Thus Baudin's sheaf formula and the $\infty$-morphism $\gamma$ induce the same maps of mapping spaces out of every test object $A$; we use the Baudin citation as an identification of this $\pi_0$-level sheaf form, not as a construction of $\gamma$.
Yoneda in the functor $\infty$-category, as in the proof of \Cref{prop:finite-duality}, identifies them as morphisms in $\Fun(\Dbc(X)^{\op},\Dqc(X))$, hence in $\Fun(\Dbc(X)^{\op},\Dbc(X))$.
\item
For an injective complex $I$ representing $(\omega,u)$ as above, each term of $I$ is injective, the unit condition holds on $I$, and the sheaf isomorphisms commute with the differentials and with degree-zero chain maps.
They therefore give a natural isomorphism of relative functors on bounded complexes, which models the map $\kappa_M$ of \textup{(1)} after passage to $\Dbc$, yielding \eqref{eq:trace-formula-complex}.
\end{enumerate}
\end{proof}

\begin{remark}\label{rem:sec3-no-baudin}
The proofs of \Cref{thm:intrinsic,thm:bidual-leq} use finite Grothendieck duality and \Cref{prop:reversal}.
Citations of Baudin in \Cref{rem:unit-choice,prop:trace-formula} compare $\pi_0$-level formulas with \cite[Lemma 3.4.1; Construction 4.1.1]{Bau}; they are not input to those proofs.
\Cref{thm:modules} enhances Baudin's derived Hom duality \cite[Theorem 4.2.7]{Bau}.
When $X$ has affine diagonal, \Cref{thm:comparison} identifies $D_\omega^{\mathrm{LEq}}$ with $\mathbb D_\omega$, and passage to homotopy categories recovers that triangulated duality, including its involutivity via \Cref{thm:bidual-leq}; see also \Cref{rem:intro-no-baudin}.
\end{remark}

\section{Enhancement of the coherent module duality}\label{sec:modules}
We now construct the functor on derived abelian categories for every Noetherian $F$-finite scheme.
This construction is also needed to compare the intrinsic statement with Baudin's functors and to pass to crystals.

\begin{construction}\label{cons:hom}
Let $X$ be Noetherian and $F$-finite, and let $\omega$ be a unit dualizing complex.
Choose a bounded below injective complex $I$ representing $\omega$ in the Grothendieck abelian category $\Fil\Coh^{C,\mathrm{unit}}_X$ of \Cref{def:unit-dualizing}, using \cite[Corollary 3.4.9]{Bau}.
The underlying $\cO_X$-complex is $R\iota(\omega)$.
Its terms are injective $\cO_X$-modules \cite[Corollary 3.4.10]{Bau}.
For $M\in\Ch^b(\Coh^F_X)$, put
\begin{equation}\label{eq:hom-complex}
 H_I(M)^n=\prod_a\shHom_{\cO_X}(M^a,I^{a+n}),\qquad
 d(h)=d_Ih-(-1)^nhd_M.
\end{equation}
Equip each term with the degreewise formula \eqref{eq:trace-formula-complex} of \Cref{prop:trace-formula}.
\end{construction}

\begin{lemma}\label{lem:relative-hom}
Let $X$ be Noetherian and $F$-finite, and let $\omega$ be a unit dualizing complex.
\Cref{cons:hom} defines an ordinary contravariant functor
\[
 H_I:\Ch^b(\Coh^F_X)^{\op}
           \longrightarrow\Ch^+(\QCoh^C_X)
\]
which sends quasi-isomorphisms to quasi-isomorphisms.
It induces an exact functor of $\infty$-categories
\[
 H_I:\Db(\Coh^F_X)^{\op}\longrightarrow
                     \mathcal D^b_{\mathrm{coh}}(\QCoh^C_X).
\]
\end{lemma}
\begin{proof}
The target $\infty$-category $\mathcal D(\QCoh^C_X)$ is the localization of $\mathrm N(\Ch(\QCoh^C_X))$ at quasi-isomorphisms \cite[Proposition 1.3.5.15]{HA}.
An ordinary functor sending quasi-isomorphisms to quasi-isomorphisms therefore induces, by the universal property of $\infty$-localization \cite[\href{https://kerodon.net/tag/01MP}{Tag 01MP}]{Kerodon}, a functor
\[
 \mathrm N(\Ch^b(\Coh^F_X))[W^{-1}]^{\op}
 \longrightarrow\mathcal D(\QCoh^C_X).
\]
The source is identified with $\Db(\Coh^F_X)^{\op}$ by \Cref{lem:bounded-stable} and \eqref{eq:relative-model}.
Opposition is an autoequivalence of $\Cat_\infty$, so it preserves localizations.
The factorization is a functor of $\infty$-categories.
Hom sends the zero complex to zero and sends a mapping cone to the corresponding shifted Hom cone, hence to a fiber, by the ordinary Hom-complex computation.
The cone description of cofibers \cite[Remark 1.3.2.17]{HA} and the exactness of the Verdier quotient in \Cref{lem:bounded-stable} prove exactness after localization.
Finally, the underlying derived Hom has bounded coherent cohomology by \cite[Lemma 4.2.3]{Bau}, applied to the pairing of \cite[Proposition 4.1.3]{Bau}, and therefore lands in $\mathcal D^b_{\mathrm{coh}}(\QCoh^C_X)$.
\end{proof}

\begin{theorem}\label{thm:modules}
Let $X$ be Noetherian and $F$-finite, and let $\omega$ be a unit dualizing complex.
The composite
\[
 \mathbb D_\omega=J_C^{-1}H_I:\Db(\Coh^F_X)^{\op}\longrightarrow\Db(\Coh^C_X)
\]
is an exact equivalence.
There is a commutative square of homotopy categories
\[
\begin{tikzcd}[column sep=large]
 h\Db(\Coh^F_X)^{\op}\ar[r,"h\mathbb D_\omega"]\ar[d,"\simeq"']&
 h\Db(\Coh^C_X)\ar[d,"\simeq"]\\
 \mathrm D^b(\Coh^F_X)^{\op}\ar[r,"{\mathrm R\mathcal Hom(-,\omega)}"']&
 \mathrm D^b(\Coh^C_X),
\end{tikzcd}
\]
where the lower functor is \cite[Definition 4.2.5]{Bau}.
\end{theorem}
\begin{proof}
The inverse $J_C^{-1}$ exists by \Cref{prop:realization}, so the displayed composite is a defined exact $\infty$-functor.
On bounded complexes it is the injective Hom construction of \cite[Definition 4.2.5]{Bau}.
This proves the stated comparison square.
The lower functor is an equivalence by the first displayed statement of \cite[Theorem 4.2.7]{Bau}.
That statement assumes only that $X$ is Noetherian and $F$-finite with a unit dualizing complex; it does not require an affine diagonal.
The vertical functors are the identifications of homotopy categories of \Cref{lem:bounded-stable}.
Hence the homotopy functor of $\mathbb D_\omega$ is an equivalence.
Apply \Cref{lem:exact-detection} to this exact functor.
The proof of the cited first statement in \cite[Theorem 4.2.7]{Bau} uses the evaluation map of \cite[Lemma 4.2.6]{Bau} together with \cite[Lemma V.1.2; Corollary V.2.3]{Har}; we invoke that argument only through the triangulated equivalence of that theorem, after the exact $\infty$-functor has already been constructed.
\end{proof}

\begin{remark}[Resolutions and variance]\label{rem:resolution}
Let $X$ be Noetherian and $F$-finite, and let $\omega$ be a unit dualizing complex.
The complex $I$ is only required to be bounded below; finite injective dimension in the abelian category of unit Cartier modules is not assumed.
Moreover, the inclusion $\iota:\Fil\Coh^{C,\mathrm{unit}}_X\hookrightarrow\Mod(\cO_X)$ is generally not exact \cite[Remark 3.4.12]{Bau}, which is why the underlying $\cO_X$-complex of $I$ in \Cref{cons:hom} is $R\iota(\omega)$.
The finiteness of the coherent source also matters: infinite products in $\Mod(\cO_X)$ of quasi-coherent sheaves need not be quasi-coherent, as already observed in \cite[proof of Proposition 4.1.3]{Bau}.
Finally, evaluation $\RHom(M,N)\otimes M\to N$ has mixed variance in $M$: the factor $\RHom(M,N)$ is contravariant in $M$, while the tensor factor $M$ is covariant.
Consequently $(M,N)\mapsto\RHom(M,N)\otimes M$ is not a functor $\Dqc(X)^{\op}\times\Dqc(X)\to\Dqc(X)$, so the evaluation map cannot be an ordinary natural transformation of functors on that product with the displayed tensor expression as source.
The independent test variable in \Cref{prop:finite-duality} and the ordinary functor of \Cref{lem:relative-hom} give constructions with the required variance.
\end{remark}

\subsection{Comparison with derived lax equalizers}
\begin{theorem}\label{thm:comparison}
Assume that $X$ is Noetherian, $F$-finite, and has affine diagonal, and let $\omega$ be a unit dualizing complex.
There are exact equivalences, commuting with the underlying-complex functors,
\begin{align*}
 \Db(\Coh^C_X)&\simeq\Cart(\Dbc(X),F_*),\\
 \Db(\Coh^F_X)&\simeq\Frob(\Dbc(X),F_*).
\end{align*}
With the unit data induced by $I$, there is a natural equivalence making the square
\[
\begin{tikzcd}[column sep=large]
 \Db(\Coh^F_X)^{\op}\ar[r,"\mathbb D_\omega"]\ar[d,"\simeq"']&
 \Db(\Coh^C_X)\ar[d,"\simeq"]\\
 \Frob(\Dbc(X),F_*)^{\op}\ar[r,"D^{\mathrm{LEq}}_\omega"']&\Cart(\Dbc(X),F_*)
\end{tikzcd}
\]
commute in $\Catbig_\infty$, where the vertical arrows are the equivalences above and the lower functor is that of \Cref{thm:intrinsic}.
\end{theorem}
\begin{remark}[Affine diagonal]\label{rem:affine-diagonal}
Having affine diagonal sits between quasi-separated and separated: it is stronger than quasi-separated and weaker than separated.
A scheme has affine diagonal if and only if the intersection of any two affine opens is affine \cite[\href{https://stacks.math.columbia.edu/tag/0BDX}{Tag 0BDX}]{SP}.
The affine-diagonal hypothesis is required for the Frobenius comparison \cite[Theorem 5.7]{MWF} (and for the geometric setup of \cite[Definition 4.10]{MWF} used in the proof).
The Cartier comparison \cite[Theorem 5.1; Corollary 6.2]{MWC} does not need an affine diagonal.
We retain the hypothesis in \Cref{thm:comparison} because both sides are used to identify $\mathbb D_\omega$ with $D^{\mathrm{LEq}}_\omega$.
\end{remark}
\begin{proof}
Since $X$ is Noetherian it is quasi-compact, so it is geometric in the terminology of \cite[Definition 4.10]{MWF}.
The derived comparison theorems are
\begin{align*}
 \mathcal D(\QCoh^C_X)&\simeq
       \Cart(\mathcal D(\QCoh(X)),F_*),
       &&\text{\cite[Theorem 5.1; Corollary 6.2]{MWC}},\\
 \mathcal D(\QCoh^F_X)&\simeq
       \Frob(\mathcal D(\QCoh(X)),F_*),
       &&\text{\cite[Theorem 5.7]{MWF}}.
\end{align*}
They are $t$-exact and their construction commutes with the forgetful functors, as specified in those theorems.
The functor $F_*$ preserves $\Dbc$ by \Cref{prop:finite-duality}, using \cite[\href{https://stacks.math.columbia.edu/tag/01Y6}{Tag 01Y6}]{SP} and exactness of $F_*$.
The defining pullback of \Cref{def:leq} therefore identifies $\Cart(\Dbc(X),F_*)$ with the full subcategory of $\Cart(\mathcal D(\QCoh(X)),F_*)$ on objects whose underlying complex lies in $\Dbc(X)$, once $\Dbc(X)$ is identified with the bounded coherent subcategory of $\mathcal D(\QCoh(X))$ by the functor $\Phi_X$ of the next paragraph, and likewise for $\Frob$.

Next we compare $\mathcal D(\QCoh(X))$ with $\Dqc(X)$.
For each affine open $U=\Spec R\subset X$, derived global sections give an exact functor $\QCoh(X)\to\operatorname{Mod}_R$, hence a functor $\mathcal D(\QCoh(X))\to\mathcal D(R)$.
These functors are natural in $U$, so the universal property of \Cref{def:dqc} gives a functor
\[
 \Phi_X:\mathcal D(\QCoh(X))\longrightarrow\Dqc(X)
\]
in $\Catbig_\infty$, natural in open immersions.
If $X=\Spec R$, then $\QCoh(X)\simeq\operatorname{Mod}_R$ as abelian categories \cite[\href{https://stacks.math.columbia.edu/tag/01IB}{Tag 01IB}]{SP}, so $\Phi_X$ is the equivalence of \Cref{lem:dqc-affine} composed with \cite[Theorem 7.1.2.13]{HA}.
Write $\mathcal D^+(\QCoh(X))$ and $\Dqc^+(X)$ for the full subcategories of cohomologically bounded-below objects in $\mathcal D(\QCoh(X))$ and $\Dqc(X)$, respectively.
The restriction
\[
 \Phi_X^+:\mathcal D^+(\QCoh(X))\longrightarrow\Dqc^+(X)
\]
is a morphism of Zariski sheaves of $\infty$-categories on a geometric $X$: the target sheaf $U\mapsto\Dqc^+(U)$ is obtained by restricting the sheaf of \Cref{lem:dqc-zariski}, and the source sheaf $U\mapsto\mathcal D^+(\QCoh(U))$ is a sheaf by \cite[Lemma 4.14]{MWF}, where $\mathcal D^+$ in that reference denotes homologically bounded-above objects.
Since $X$ is qcqs and $\Phi_X^+$ is an equivalence on every affine, \cite[Lemma 4.5]{MWF} implies that $\Phi_X^+$ is an equivalence on $\mathcal D^+(\QCoh(X))$.
The full subcategory $\Dbc(X)$ of $\Dqc(X)$ is identified with the bounded coherent subcategory of $\mathcal D(\QCoh(X))$: coherence of cohomology is local on a finite affine cover \cite[\href{https://stacks.math.columbia.edu/tag/01XZ}{Tag 01XZ}]{SP}, and vanishing of a cohomology sheaf is local on an open cover.
Now apply \Cref{prop:realization}.

The two composites $\Phi_X\circ\mathcal D(F_*)$ and $F_*\circ\Phi_X$ become, after each projection $\Dqc(X)\to\mathcal D(R)$, the same functor $F_*\circ\Gamma(U,-)$: affine Frobenius is restriction of scalars, and $\Phi_X$ is built from derived global sections.
Uniqueness of morphisms into a limit in $\Catbig_\infty$ therefore gives an equivalence
\[
 \Phi_X\circ\mathcal D(F_*)\simeq F_*\circ\Phi_X
\]
in $\Fun(\mathcal D(\QCoh(X)),\Dqc(X))$, and likewise on the Cartier side.
Transport of the defining pullback of \Cref{def:leq} along this natural equivalence identifies $\Cart(\Dbc(X),F_*)$ with the corresponding full subcategory of $\Cart(\mathcal D(\QCoh(X)),F_*)$.

It remains to identify $\mathbb D_\omega$ with $D^{\mathrm{LEq}}_\omega$ as morphisms in
\[
 \Fun\bigl(\Db(\Coh^F_X)^{\op},\,\Db(\Coh^C_X)\bigr),
\]
after composing with the equivalences $\Db(\Coh^F_X)\simeq\Frob(\Dbc(X),F_*)$ and $\Db(\Coh^C_X)\simeq\Cart(\Dbc(X),F_*)$ already established.
By \Cref{def:leq}, a functor $\mathcal E\to\Cart(\Dbc(X),F_*)$ is a pair consisting of an underlying functor $U:\mathcal E\to\Dbc(X)$ and a morphism $\kappa:F_*U\to U$ in $\Fun(\mathcal E,\Dbc(X))$.
Write $\mathcal E=\Db(\Coh^F_X)^{\op}$.
After the displayed equivalences, $D^{\mathrm{LEq}}_\omega$ is such a pair, with underlying functor
\[
 U_{\mathrm{LEq}}=\RHom(-,\omega)
\]
and Cartier arrow $\kappa=D_\omega(\tau)\circ\gamma$.
The functor $\mathbb D_\omega=J_C^{-1}H_I$ likewise determines a pair, with underlying functor
\[
 U_{\mathrm{mod}}=J_C^{-1}H_I
\]
and Cartier arrow the localization of the degreewise pairing of \cite[Construction 4.1.1]{Bau} on \eqref{eq:hom-complex}.

The underlying functors $U_{\mathrm{LEq}}$ and $U_{\mathrm{mod}}$ are identified in $\Fun(\mathcal E,\Dbc(X))$ as follows.
After the identification of $\Dbc$ with the bounded coherent subcategory of $\Dqc$ via $\Phi_X$, both values $U_{\mathrm{LEq}}(M)$ and $U_{\mathrm{mod}}(M)$ are the value at $\omega$ of a right adjoint to $-\otimes M$, representing the functor $A\mapsto\Map(A\otimes M,\omega)$.
\begin{enumerate}
\item
On the one hand, $U_{\mathrm{LEq}}(M)=\RHom(M,\omega)$ is that value by the tensor--Hom adjunction $(-\otimes M)\dashv\RHom(M,-)$ of \Cref{lem:dqc-properties}: $\Map(A\otimes M,\omega)\simeq\Map\bigl(A,\RHom(M,\omega)\bigr)$.
\item
On the other hand, $U_{\mathrm{mod}}(M)$ is computed by Hom into the bounded-below injective complex $I$ of \Cref{cons:hom}.
On each affine $U=\Spec R$, \Cref{lem:dqc-affine} identifies $\Dqc(U)$ with $\cD(R)$.
By \cite[Proposition 1.3.5.15]{HA}, the $\infty$-category $\cD(R)$ is obtained by localizing the ordinary nerve of chain complexes at quasi-isomorphisms.
By \cite[Proposition 1.3.5.6]{HA}, a bounded-below complex of injectives is fibrant in the injective model structure.
By \cite[Lemma 1.3.5.11; Proposition 1.3.5.14]{HA}, chain-complex Hom into a fibrant complex sends quasi-isomorphisms in the source to quasi-isomorphisms of Hom complexes; those lemmas concern Hom complexes, not mapping spaces in $\cD(R)$.
Mapping spaces in $\cD(R)$ are those of the differential graded nerve: \cite[Remark 1.3.1.12]{HA} identifies them with the Dold--Kan construction on the truncated Hom complex, and \cite[Proposition 1.3.1.17]{HA} compares that nerve with the nerve of the associated simplicial category.
Thus on affines, $H_I(M)$ computes $\Map(A\otimes M,\omega)$ for test objects $A$, hence the same value at $\omega$ of a right adjoint to $-\otimes M$.
\end{enumerate}
Uniqueness of right adjoints \cite[Remark 5.2.2.2]{HTT} gives $U_{\mathrm{LEq}}\simeq U_{\mathrm{mod}}$ in $\Fun$.

The Cartier arrows then match: \Cref{prop:trace-formula} identifies $\gamma$ with \cite[Lemma 3.4.1, first item and (3.4.1.c)]{Bau} in $\Fun(\Dbc(X)^{\op},\Dqc(X))$, and \cite[Construction 4.1.1]{Bau} is that isomorphism composed with $D_\omega(\tau)$.
The localization universal property \cite[\href{https://kerodon.net/tag/01MP}{Tag 01MP}]{Kerodon} applied to the ordinary natural transformation of \cite[Construction 4.1.1]{Bau} on bounded complexes produces this composite $D_\omega(\tau)\circ\gamma$ as a morphism in $\Fun$.
The two pairs $(U,\kappa)$ therefore coincide as morphisms into the cospan defining $\Cart(\Dbc(X),F_*)$.
The pullback universal property of \Cref{def:leq} gives an equivalence in $\Fun(\mathcal E,\Cart(\Dbc(X),F_*))$, and hence in $\Fun(\mathcal E,\Db(\Coh^C_X))$ after postcomposition with the equivalence $\Cart(\Dbc(X),F_*)\simeq\Db(\Coh^C_X)$.
\end{proof}

\section{Stable categories of crystals}\label{sec:crystals}
Let $X$ be Noetherian and $F$-finite.
Let $\Nil^F_X\subset\Coh^F_X$ and $\Nil^C_X\subset\Coh^C_X$ consist of the modules on which an iterate of the structure map is zero.
For Cartier modules the $n$-fold iterate of the structure map is
\[
 \kappa^{n}=\kappa\circ F_*\kappa\circ\cdots\circ F_*^{n-1}\kappa;
\]
for Frobenius modules it is
\[
 \tau^{n}=F_*^{n-1}(\tau)\circ\cdots\circ\tau:M\to F_*^nM.
\]
Here the superscript $n$ indexes that $n$-fold iterate of $\kappa$ or $\tau$; it is not a cohomological degree.
The respective locally nilpotent quasi-coherent subcategories are denoted $\LNil^F_X$ and $\LNil^C_X$.
They define the abelian categories
\[
 \Crys^F_X=\Coh^F_X/\Nil^F_X,\qquad
 \Crys^C_X=\Coh^C_X/\Nil^C_X,
\]
\[
 \QCrys^F_X=\QCoh^F_X/\LNil^F_X,\qquad
 \QCrys^C_X=\QCoh^C_X/\LNil^C_X
\]
of \cite[Definition 2.2.8; Definitions 3.3.1 and 3.3.4]{Bau}.
The quasi-crystal categories are Grothendieck abelian \cite[Lemma 2.2.10; Lemma 3.3.5]{Bau}.

\begin{definition}\label{def:nil-complex}
Let $X$ be Noetherian and $F$-finite.
Let
\[
 \cN_F(X)\subset\Db(\Coh^F_X),\qquad
 \cN_C(X)\subset\Db(\Coh^C_X)
\]
be the full subcategories of objects all of whose cohomology modules are nilpotent.
A map $f$ is a nil-quasi-isomorphism if the kernel and cokernel of each $H^i(f)$ are nilpotent; let $W_F$ and $W_C$ be these classes on the Frobenius and Cartier sides, respectively.
Let
\[
 \KF(X)\subset\mathcal D(\QCrys^F_X),\qquad
 \KC(X)\subset\mathcal D(\QCrys^C_X)
\]
be the full subcategories of objects with bounded cohomology lying in the essential images of
\[
 \Crys^F_X\to\QCrys^F_X,\qquad
 \Crys^C_X\to\QCrys^C_X,
\]
respectively.
This is not a priori the essential image of $\Db(\Crys^F_X)$ or of $\Db(\Crys^C_X)$; that identification is \Cref{prop:quotient-comparison}.
\end{definition}

\begin{definition}[{cf. \cite[Lecture 26, Definition 1]{Chromatic}}]\label{def:thick}
We say a full subcategory of a stable $\infty$-category is \textit{thick} if it is stable and closed under retracts.
\end{definition}

\begin{lemma}\label{lem:nil-stable}
Let $X$ be Noetherian and $F$-finite.
\begin{enumerate}
\item
The subcategories $\cN_F(X)$ and $\cN_C(X)$ are thick.
\item
A map $f$ lies in $W_F$ (resp.\ $W_C$) if and only if $\cofib(f)\in\cN_F(X)$ (resp.\ $\cN_C(X)$).
\item
The categories $\KF(X)$ and $\KC(X)$ are stable.
\end{enumerate}
\end{lemma}
\begin{proof}
The Frobenius and Cartier cases are parallel.
\begin{enumerate}
\item
Nilpotent coherent objects form a Serre subcategory: if $\tau^{r}=0$ (resp.\ $\kappa^{r}=0$), the same vanishing holds for subobjects and quotients, and if two objects in a short exact sequence satisfy $\tau^{r}=0$ and $\tau^{s}=0$ (resp.\ $\kappa^{r}=0$ and $\kappa^{s}=0$), the middle term satisfies $\tau^{r+s}=0$ (resp.\ $\kappa^{r+s}=0$).
For the standard $t$-structure, a cofiber sequence gives in each degree an exact sequence
\[
 0\longrightarrow\operatorname{coker}H^i(f)
 \longrightarrow H^i(\cofib f)
 \longrightarrow\ker H^{i+1}(f)\longrightarrow0.
\]
This is the cohomology sequence of the mapping-cone model for the cofiber, after passage to the homotopy category of \Cref{lem:bounded-stable}; the mapping cone is a cofiber by \cite[Remark 1.3.2.17]{HA}.
The Serre property therefore closes $\cN_F(X)$ and $\cN_C(X)$ under finite cofibers and shifts.
Stability of these full subcategories follows from \cite[Lemma 1.1.3.3]{HA}.
Cohomology preserves retract diagrams and nilpotence passes to retracts, so $\cN_F(X)$ and $\cN_C(X)$ are closed under retracts, hence thick in the sense of \Cref{def:thick}.
\item
The same exact sequence and the Serre property show that $f$ lies in $W_F$ if and only if $\cofib(f)\in\cN_F(X)$, and likewise for $W_C$ and $\cN_C(X)$.
\item
Stability of $\KF(X)$ and $\KC(X)$ is the case of \Cref{lem:bounded-serre} for the Serre subcategory given by the essential image of crystals inside quasi-crystals, which is \cite[Lemma 3.3.6]{Bau} on the Cartier side and \cite[Proposition 3.4.2(b)(c)]{BP} on the Frobenius side.
\end{enumerate}
\end{proof}

\begin{proposition}\label{prop:quotient-comparison}
Let $X$ be Noetherian and $F$-finite.
\begin{enumerate}
\item
Write
\[
 \Db(\Coh^F_X)/\cN_F(X)\,:=\,\Db(\Coh^F_X)[W_F^{-1}],\qquad
 \Db(\Coh^C_X)/\cN_C(X)\,:=\,\Db(\Coh^C_X)[W_C^{-1}]
\]
for the localizations of \cite[Theorem I.3.3(i)]{NS}; they are stable and the quotient functors are exact.
\item
Termwise application of the abelian quotients $\QCoh^F_X\to\QCrys^F_X$ and $\QCoh^C_X\to\QCrys^C_X$ induces exact equivalences
\begin{equation}\label{eq:crystal-quotient}
 \Db(\Coh^F_X)/\cN_F(X)
 \xrightarrow{\ \simeq\ }\KF(X),\qquad
 \Db(\Coh^C_X)/\cN_C(X)
 \xrightarrow{\ \simeq\ }\KC(X).
\end{equation}
\item
The functors $\Db(\Crys^F_X)\to\KF(X)$ and $\Db(\Crys^C_X)\to\KC(X)$ induced by inclusion are also exact equivalences.
\end{enumerate}
\end{proposition}
\begin{proof}
The Frobenius and Cartier cases are parallel.
\begin{enumerate}
\item
By \Cref{lem:nil-stable}, $W_F$ (resp.\ $W_C$) is the class of morphisms whose cofiber lies in $\cN_F(X)$ (resp.\ $\cN_C(X)$), so \cite[Theorem I.3.3(i)]{NS} applies to the small stable $\infty$-categories $\Db(\Coh^F_X)$ and $\Db(\Coh^C_X)$ of \Cref{def:bounded}.
This gives the stated localizations, their stability, and the exactness of the quotient functors.
\item
Let $q_F:\QCoh^F_X\to\QCrys^F_X$ and $q_C:\QCoh^C_X\to\QCrys^C_X$ be the exact abelian quotient functors.
Termwise application gives ordinary functors $\Ch^b(\Coh^F_X)\to\Ch(\QCrys^F_X)$ and $\Ch^b(\Coh^C_X)\to\Ch(\QCrys^C_X)$ sending quasi-isomorphisms to quasi-isomorphisms, because $q_F$ and $q_C$ are exact.
They also send mapping cones to mapping cones, because $q_F$ and $q_C$ are additive and exact.
Localizing the sources by \Cref{lem:bounded-stable} and the targets by \cite[Propositions 1.3.5.13 and 1.3.5.15]{HA}, the localization universal property \cite[\href{https://kerodon.net/tag/01MP}{Tag 01MP}]{Kerodon} produces $\infty$-functors $\Db(\Coh^F_X)\to\mathcal D(\QCrys^F_X)$ and $\Db(\Coh^C_X)\to\mathcal D(\QCrys^C_X)$.
The cone description of cofibres \cite[Remark 1.3.2.17]{HA} and the exactness of the Verdier quotient in \Cref{lem:bounded-stable} prove exactness after localization.
On representing complexes one has $H^i(q_F(M))=q_F H^i(M)$ and $H^i(q_C(M))=q_C H^i(M)$, so the functors send $\Db(\Coh^F_X)$ into $\KF(X)$ and $\Db(\Coh^C_X)$ into $\KC(X)$ as defined in \Cref{def:nil-complex}, and kill $\cN_F(X)$ and $\cN_C(X)$ respectively.
The exact universal property of \cite[Theorem I.3.3(i)]{NS} now gives the functors in \eqref{eq:crystal-quotient}.

Write $L_F$ and $L_C$ for these functors.
Their homotopy functors are the classical localization comparisons: the homotopy category of an $\infty$-localization is the localization of the homotopy category \cite[\href{https://kerodon.net/tag/01MV}{Tag 01MV}]{Kerodon}.
The sources $h(\Db(\Coh^F_X)[W_F^{-1}])$ and $h(\Db(\Coh^C_X)[W_C^{-1}])$ are identified with Baudin's $S^{-1}D^b_{\mathrm{coh}}(\QCoh^F)$ and $S^{-1}D^b_{\mathrm{coh}}(\QCoh^C)$ by \Cref{prop:realization}.
Nil-quasi-isomorphisms are preserved and reflected because both comparison functors are equivalences which detect those maps on cohomology, by the same argument as \cite[Lemma 4.3.4]{Bau}, using the $\QCoh$ comparison of \cite[Corollary 3.2.24]{Bau} on the Cartier side.
On the Cartier side the resulting comparison is an equivalence by \cite[Corollary 3.3.9]{Bau}.
On the Frobenius side it is \cite[Theorem 5.3.1(b)(c)]{BP}, as recorded for Noetherian schemes in \cite[proof of Proposition 2.2.15]{Bau}; that argument uses Noetherianity rather than the separated-scheme conventions of the book.
Both source and target are stable and the functors are exact, so \Cref{lem:exact-detection} proves the $\infty$-categorical equivalences.
\item
Termwise inclusion similarly constructs the exact functors $\Db(\Crys^F_X)\to\KF(X)$ and $\Db(\Crys^C_X)\to\KC(X)$, again by sending mapping cones to mapping cones and applying \cite[\href{https://kerodon.net/tag/01MP}{Tag 01MP}]{Kerodon} together with \cite[Remark 1.3.2.17]{HA}.
Their homotopy functors are equivalences by \cite[Corollary 3.3.8]{Bau} on the Cartier side and by \cite[Proposition 2.2.15]{Bau} on the Frobenius side.
Apply \Cref{lem:exact-detection} once more.
\end{enumerate}
\end{proof}

\begin{lemma}\label{lem:nil-heart-generation}
Let $X$ be Noetherian and $F$-finite.
Let $M\in\cN_F(X)$ (resp.\ $\cN_C(X)$), and suppose the cohomology of $M$ is concentrated in degrees $[a,b]$.
Then $M$ may be rebuilt from the shifted heart objects
\[
 H^i(M)[-i]\in\Nil^F_X
 \qquad\text{(resp.\ $\Nil^C_X$)},
\]
for $a\le i\le b$, by finitely many fibre sequences of standard truncations.
In particular, $\cN_F(X)$ (resp.\ $\cN_C(X)$) is the smallest stable subcategory of $\Db(\Coh^F_X)$ (resp.\ $\Db(\Coh^C_X)$) containing $\Nil^F_X$ (resp.\ $\Nil^C_X$).
\end{lemma}
\begin{proof}
The standard $t$-structures on $\mathcal D(\QCoh^F_X)$ and $\mathcal D(\QCoh^C_X)$ of \Cref{lem:grothendieck-derived} restrict to $\Db(\Coh^F_X)$ and $\Db(\Coh^C_X)$ along \Cref{prop:realization}.
By \cite[Definition 1.2.1.1]{HA}, a $t$-structure determines truncation functors.
By \cite[Proposition 1.2.1.5]{HA}, these truncations fit into fibre sequences.
The fibre sequence
\[
 \tau_{\geq n}M\longrightarrow M\longrightarrow\tau_{\leq n-1}M
\]
is recorded in \cite[Remark 1.2.1.8]{HA}.
(Here we use cohomological indexing; the relation to HA's homological convention is $\cC^{\leq n}=\cC^{\mathrm{HA}}_{\geq -n}$, as in \Cref{sec:perverse}.)
Truncations preserve bounded coherent cohomology and, by the Serre property, nilpotent cohomology, so they preserve $\cN_F(X)$ and $\cN_C(X)$.
If $M$ has cohomology concentrated in degrees $[a,b]$, the fibre of $\tau_{\geq a+1}M\to M$ is equivalent to $H^a(M)[-a]$, and inductively each successive truncation step adjoins one shifted cohomology object $H^i(M)[-i]$ until the amplitude is exhausted.
This rebuilds $M$ from those shifted heart objects.
The last sentence follows because a stable subcategory is closed under fibres and shifts \cite[Lemma 1.1.3.3]{HA}.
\end{proof}

\begin{theorem}\label{thm:crystals}
Let $X$ be Noetherian and $F$-finite, and let $\omega$ be a unit dualizing complex.
\begin{enumerate}
\item
The equivalence $\mathbb D_\omega$ of \Cref{thm:modules} restricts to an exact equivalence
\[
 \mathbb D_\omega:\cN_F(X)^{\op}\xrightarrow{\ \simeq\ }\cN_C(X),
\]
and induces an exact equivalence
\[
 \mathbb D^{\mathrm{crys}}_\omega:\KF(X)^{\op}\simeq\KC(X).
\]
\item
Under \Cref{prop:quotient-comparison}, there is a natural equivalence making the square
\[
\begin{tikzcd}[column sep=large]
 \Db(\Coh^F_X)^{\op}\ar[r,"\mathbb D_\omega"]\ar[d,"L_F^{\op}"']&
 \Db(\Coh^C_X)\ar[d,"L_C"]\\
 \KF(X)^{\op}\ar[r,"\mathbb D^{\mathrm{crys}}_\omega"']&\KC(X)
\end{tikzcd}
\]
commute in $\Catbig_\infty$, where $L_F$ and $L_C$ denote the localization functors of \eqref{eq:crystal-quotient}.
\item
On homotopy categories the lower equivalence is \cite[Theorem 4.3.5]{Bau}.
\end{enumerate}
\end{theorem}
\begin{proof}
\begin{enumerate}
\item
Membership in $\cN_F(X)$ and in $\cN_C(X)$ is a condition on cohomology, hence is detected on the homotopy category.
First let $M$ be a coherent nilpotent Frobenius module with $\tau_M^{r}=0$.
Iteration of \eqref{eq:trace-formula} shows that every term of $H_I(M)$ has zero $r$th Cartier iterate.
Its cohomology is therefore nilpotent by the Serre property, so $\mathbb D_\omega(M)\in\cN_C(X)$.
For a coherent nilpotent Cartier module, the inverse homotopy functor is the pairing of \cite[Construction 4.1.4; Proposition 4.1.6]{Bau}.
In a unit target, its $r$th Frobenius iterate is obtained by applying $(F^\flat)^r$ to the Hom map and composing with the adjoint of $\kappa_M^{r}$ and the inverse unit isomorphism.
It vanishes when $\kappa_M^{r}=0$.
The same bounded Hom model thus shows that $h\mathbb D_\omega^{-1}$ carries such a heart object into $\cN_F(X)$.

By \Cref{lem:nil-heart-generation}, every object of $\cN_F(X)$ (resp.\ $\cN_C(X)$) is rebuilt from shifts of heart objects in $\Nil^F_X$ (resp.\ $\Nil^C_X$) by finitely many truncation fibre sequences.
Since $h\mathbb D_\omega$ and $h\mathbb D_\omega^{-1}$ preserve those heart objects, exactness of the ambient equivalence $\mathbb D_\omega$ and stability of $\cN_F(X)$ and $\cN_C(X)$ imply that $\mathbb D_\omega$ restricts to an exact equivalence $\cN_F(X)^{\op}\simeq\cN_C(X)$.
In particular $\mathbb D_\omega$ preserves and reflects $W_F$ and $W_C$, by \Cref{lem:nil-stable}.
The localization universal property of \cite[Theorem I.3.3(i)]{NS} therefore gives an exact equivalence of the localizations.
Use \Cref{prop:quotient-comparison} to identify the localized categories with $\KF(X)$ and $\KC(X)$, obtaining $\mathbb D^{\mathrm{crys}}_\omega$.
\item
The same localization universal property gives a natural equivalence making the displayed square commute in $\Catbig_\infty$.
\item
The same universal property after taking homotopy categories identifies the lower functor with Baudin's localized pairing \cite[Theorem 4.3.5]{Bau}.
\end{enumerate}
\end{proof}

\begin{remark}[Mapping spaces of crystals]\label{rem:crystal-maps}
The proof of \Cref{prop:quotient-comparison} uses stability and exactness before the homotopy-category comparison.
An equivalence of homotopy categories alone would not suffice.
More explicitly, \cite[Theorem I.3.3(ii)]{NS} computes the quotient mapping spaces by filtered colimits
\[
 \Map_{\Db(\Coh^F_X)/\cN_F}(M,N)
 \simeq\mathop{\operatorname{colim}}_{(Z\to N)\in\cN_F/N}
            \Map_{\Db(\Coh^F_X)}(M,\cofib(Z\to N)),
\]
\[
 \Map_{\Db(\Coh^C_X)/\cN_C}(M,N)
 \simeq\mathop{\operatorname{colim}}_{(Z\to N)\in\cN_C/N}
            \Map_{\Db(\Coh^C_X)}(M,\cofib(Z\to N)).
\]
The colimits are taken in the $\infty$-category of spaces, indexed by the $\infty$-categorical slices $\cN_F/N$ and $\cN_C/N$, which are small after the replacement of \Cref{def:bounded}.
Thus the quotients give the higher mapping information as well as the usual categories of fractions.
\end{remark}

\section{Perverse \texorpdfstring{$t$}{t}-structures and normalization}\label{sec:perverse}
All $t$-structures in this section use cohomological indexing.
For an exact contravariant equivalence $D:\cC^{\op}\simeq\cD$, the structure transported from $(\cD^{\leq0},\cD^{\geq0})$ has
\begin{equation}\label{eq:transport}
 {}^D\cC^{\leq0}=\{M:DM\in\cD^{\geq0}\},\qquad
 {}^D\cC^{\geq0}=\{M:DM\in\cD^{\leq0}\}.
\end{equation}
An equivalence of stable $\infty$-categories induces an equivalence of homotopy categories, and $h(\cC^{\op})=(h\cC)^{\op}$.
A $t$-structure on a stable $\infty$-category is a $t$-structure on its homotopy category \cite[Definition 1.2.1.4]{HA}, so \eqref{eq:transport} is a $t$-structure by transport along $hD$, using the opposite $t$-structure.
The relation to the homological convention of \cite[Warning 1.2.1.13]{HA} is
$\cC^{\leq n}=\cC^{\mathrm{HA}}_{\geq -n}$ and
$\cC^{\geq n}=\cC^{\mathrm{HA}}_{\leq -n}$.

\begin{definition}\label{def:dual-perversity}
Let $X$ be Noetherian and $F$-finite, and let $\omega$ be a unit dualizing complex.
For $x\in X$, write $d_\omega(x)$ for the unique integer such that
\[
 \omega_x\bigl[d_\omega(x)\bigr]
\]
is a normalized dualizing complex over $\cO_{X,x}$ \cite[\href{https://stacks.math.columbia.edu/tag/0AWF}{Tag 0AWF}]{SP}.
The stalk is dualizing: after passing to an affine, this is localization of a dualizing complex \cite[\href{https://stacks.math.columbia.edu/tag/0A7G}{Tag 0A7G}]{SP}.
By that same Tag~0AWF, $-d_\omega$ is a dimension function \cite[\href{https://stacks.math.columbia.edu/tag/02I9}{Tag 02I9}]{SP}: it drops by one under each immediate specialization, so $d_\omega$ rises by one, and for any specialization $x\rightsquigarrow y$ one has
\[
 d_\omega(y)-d_\omega(x)
 =\operatorname{codim}\bigl(\overline{\{y\}},\,\overline{\{x\}}\bigr)
\]
\cite[\href{https://stacks.math.columbia.edu/tag/02IA}{Tag 02IA}]{SP}.
For example, if $X$ is irreducible of dimension $n$ and $\omega$ is shifted so that $d_\omega(\xi)=-n$ at the generic point $\xi$, then
\[
 d_\omega(x)=-\dim\overline{\{x\}}
\]
for every $x\in X$; the middle-perversity normalization \eqref{eq:normalization} is exactly this choice of shift (see \Cref{cor:middle}).
We use the shift convention $H^i(K[n])=H^{i+n}(K)$.
Applied to the normalized complex $\omega_x[d_\omega(x)]$, \cite[\href{https://stacks.math.columbia.edu/tag/0A82}{Tag 0A82}]{SP} gives
\[
 R\Gamma_{\mathfrak m_x}(\omega_x)\simeq E_x[-d_\omega(x)],
\]
where $E_x$ is an injective hull of $k(x)$.
In particular $d_\omega(x)$ is the unique cohomological degree in which $R\Gamma_{\mathfrak m_x}(\omega_x)$ is nonzero.
Define
\[
 {}^{\omega}\Db(\Coh^F_X)^{\leq0}
 \qquad\text{and}\qquad
 {}^{\omega}\Db(\Coh^F_X)^{\geq0}
\]
by
\begin{align}
 M\in{}^{\omega}\Db(\Coh^F_X)^{\leq0}
 &\quad\Longleftrightarrow\quad H^j(M_x)=0
               \quad(j>d_\omega(x),\ x\in X),\label{eq:perv-stalk}\\
 M\in{}^{\omega}\Db(\Coh^F_X)^{\geq0}
 &\quad\Longleftrightarrow\quad
 H^j_{\mathfrak m_x}(M_x)=0
               \quad(j<d_\omega(x),\ x\in X),\label{eq:perv-support}
\end{align}
(The stalk cohomology in \eqref{eq:perv-stalk} versus local cohomology in \eqref{eq:perv-support} is intentional: local duality interchanges them, cf.\ \Cref{lem:local-duality}.)
Likewise define
\[
 {}^{\omega}\Db(\Coh^C_X)^{\leq0}
 \qquad\text{and}\qquad
 {}^{\omega}\Db(\Coh^C_X)^{\geq0}
\]
by the same stalk and support vanishing conditions with $\Coh^F_X$ replaced by $\Coh^C_X$.
The complexes in these formulas are the underlying $\cO_X$-complexes.
Write
\[
 \Perv_\omega(\Coh^F_X)
 \qquad\text{and}\qquad
 \Perv_\omega(\Coh^C_X)
\]
for the intersections of these two full subcategories on each side, which will be their abelian hearts.
\end{definition}

\begin{lemma}\label{lem:local-duality}
Let $X$ be Noetherian and $F$-finite, and let $\omega$ be a unit dualizing complex.
Let $M\in\Db(\Coh^F_X)$ or $M\in\Db(\Coh^C_X)$, write $N=D_\omega M$ for the underlying dual, fix $x\in X$, and set $d=d_\omega(x)$.
Let $E_x$ be an injective hull of the residue field of $\cO_{X,x}$.
There is an isomorphism of finite modules
\begin{equation}\label{eq:local-duality-test}
 \widehat{H^n(N_x)}\cong
 \Hom_{\cO_{X,x}}\bigl(H^{d-n}_{\mathfrak m_x}(M_x),E_x\bigr).
\end{equation}
In particular $H^n(N_x)=0$ if and only if $H^{d-n}_{\mathfrak m_x}(M_x)=0$.
\end{lemma}
\begin{proof}
The statement \cite[\href{https://stacks.math.columbia.edu/tag/0A84}{Tag 0A84}]{SP} applies to a normalized dualizing complex:
\[
 R\Hom_{\cO_{X,x}}(K,\omega_{\mathrm{norm}})^\wedge
 \simeq
 R\Hom_{\cO_{X,x}}\bigl(R\Gamma_{\mathfrak m_x}(K),E_x[0]\bigr).
\]
Take $\omega_{\mathrm{norm}}=\omega_x[d]$, as in \Cref{def:dual-perversity}, and $K=M_x$.
The identification $N_x\simeq R\Hom_{\cO_{X,x}}(M_x,\omega_x)$ is the stalk case of the affine comparison of internal Hom in \Cref{lem:dqc-properties,lem:dqc-affine}: localization $A\to A_{\mathfrak p}$ is flat, a bounded coherent first input is a pseudo-coherent complex \cite[\href{https://stacks.math.columbia.edu/tag/066E}{Tag 066E(2)}]{SP}, and \cite[\href{https://stacks.math.columbia.edu/tag/0A6A}{Tag 0A6A(3)}]{SP} gives the morphism.
Then $R\Hom(M_x,\omega_x[d])\simeq N_x[d]$, so
\[
 (N_x[d])^\wedge
 \simeq
 R\Hom\bigl(R\Gamma_{\mathfrak m_x}(M_x),E_x[0]\bigr).
\]
The complex $E_x[0]$ is injective, so the cohomology of the right-hand side is
$\Hom(H^{-k}_{\mathfrak m_x}(M_x),E_x)$ in degree $k$.
The cohomology of $N_x$ is finite.
The cohomology of the derived completion is the ordinary $\mathfrak m_x$-adic completion of that cohomology \cite[\href{https://stacks.math.columbia.edu/tag/0A06}{Tag 0A06}]{SP}.
Thus $H^k((N_x[d])^\wedge)\cong\widehat{H^{k+d}(N_x)}$.
Set $n=k+d$ to obtain \eqref{eq:local-duality-test}.
Faithful flatness of completion \cite[\href{https://stacks.math.columbia.edu/tag/00MC}{Tag 00MC}]{SP} detects vanishing of finite modules.
The module $E_x$ is an injective cogenerator of the category of $\mathfrak m_x$-torsion modules, so the right-hand side of \eqref{eq:local-duality-test} vanishes if and only if $H^{d-n}_{\mathfrak m_x}(M_x)$ vanishes.
\end{proof}

\begin{definition}[{Bounded $t$-structure; cf.\ \cite[1.3.3]{BBD}}]\label{def:bounded-t}
A $t$-structure $(\cC^{\leq0},\cC^{\geq0})$ on a triangulated category---or on a stable $\infty$-category, via \cite[Definition 1.2.1.4]{HA}---is \emph{bounded} if
\[
 \bigcup_{n\in\mathbb Z}\cC^{\leq n}
 =\cC
 =\bigcup_{n\in\mathbb Z}\cC^{\geq n},
\]
equivalently if every object lies in $\cC^{\geq a}\cap\cC^{\leq b}$ for some integers $a\leq b$.
\end{definition}

\begin{lemma}\label{lem:perverse-t}
Let $X$ be Noetherian and $F$-finite, and let $\omega$ be a unit dualizing complex.
The subcategories in \Cref{def:dual-perversity} are bounded $t$-structures on $\Db(\Coh^F_X)$ and $\Db(\Coh^C_X)$ in the sense of \Cref{def:bounded-t}.
\end{lemma}
\begin{proof}
The conditions \eqref{eq:perv-stalk} and \eqref{eq:perv-support} depend only on the underlying complex in $\mathcal D^b_{\mathrm{coh}}(X)$.
Let $M\in\Db(\Coh^F_X)$, set $N=\mathbb D_\omega M$ and $d=d_\omega(x)$.
By \Cref{lem:local-duality}, $H^n(N_x)=0$ for all $n>0$ if and only if $H^{j}_{\mathfrak m_x}(M_x)=0$ for all $j<d$.
Thus \eqref{eq:perv-support} holds if and only if $N$ lies in the standard $\mathrm D^{\leq0}$ on $\Db(\Coh^C_X)$.
The same lemma with $M$ and $N$ interchanged shows that $H^j(M_x)=0$ for all $j>d$ if and only if $H^i_{\mathfrak m_x}(N_x)=0$ for all $i<0$.
For bounded $N$, that vanishing at every point is the zero-perversity condition for $\mathrm D^{\geq0}$: in \cite[Theorem 6.6]{MWC} the functor $Ri^!_x$ is sections with support in $\{x\}$, as in \cite[Chapter IV, Variation 8]{Har}, hence $R\Gamma_{\mathfrak m_x}$ of the stalk, and that perversity is the standard $t$-structure on underlying complexes \cite[Lemma 6.7]{MWC}.
Thus \eqref{eq:perv-stalk} holds if and only if $N$ lies in the standard $\mathrm D^{\geq0}$.
The same argument with Frobenius and Cartier interchanged treats objects of $\Db(\Coh^C_X)$.
These are the four equalities of \Cref{thm:perverse}(1).
The standard $t$-structure on $\Db(\Coh^C_X)$ (resp.\ $\Db(\Coh^F_X)$) is a $t$-structure, so its preimage under the exact equivalence $\mathbb D_\omega$ of \Cref{thm:modules} is a $t$-structure, by \eqref{eq:transport}.
Every object of $\mathcal D^b_{\mathrm{coh}}(X)$ has bounded coherent cohomology, and $d_\omega$ takes values in a finite range on the points that meet the support of a given complex, so this $t$-structure is bounded in the sense of \Cref{def:bounded-t}.
Pulling the same conditions back along the forgetful functors $\Db(\Coh^F_X)\to\mathcal D^b_{\mathrm{coh}}(X)$ and $\Db(\Coh^C_X)\to\mathcal D^b_{\mathrm{coh}}(X)$ yields bounded $t$-structures on the module categories.
\end{proof}

\begin{theorem}\label{thm:perverse}
Let $X$ be Noetherian and $F$-finite, and let $\omega$ be a unit dualizing complex.
\begin{enumerate}
\item
The equivalence $\mathbb D_\omega$ of \Cref{thm:modules} exchanges the standard $t$-structure on one module category with the perverse $t$-structure of \Cref{lem:perverse-t} on the other, with the inequalities reversed as in \eqref{eq:transport}, i.e.,
\[
 \begin{aligned}
 {}^{\omega}\Db(\Coh^F_X)^{\leq0}
 &=\mathbb D_\omega^{-1}\bigl(\Db(\Coh^C_X)^{\geq0}\bigr),
 &
 {}^{\omega}\Db(\Coh^F_X)^{\geq0}
 &=\mathbb D_\omega^{-1}\bigl(\Db(\Coh^C_X)^{\leq0}\bigr),
 \\
 {}^{\omega}\Db(\Coh^C_X)^{\leq0}
 &=\mathbb D_\omega\bigl(\Db(\Coh^F_X)^{\geq0}\bigr),
 &
 {}^{\omega}\Db(\Coh^C_X)^{\geq0}
 &=\mathbb D_\omega\bigl(\Db(\Coh^F_X)^{\leq0}\bigr).
 \end{aligned}
\]
\item
There are anti-equivalences of abelian hearts
\[
 \Perv_\omega(\Coh^F_X)^{\op}\simeq\Coh^C_X,
 \qquad
 \Perv_\omega(\Coh^C_X)^{\op}\simeq\Coh^F_X.
\]
\item
The forgetful functors $\Db(\Coh^F_X)\to\mathcal D^b_{\mathrm{coh}}(X)$ and $\Db(\Coh^C_X)\to\mathcal D^b_{\mathrm{coh}}(X)$ are $t$-exact for these perverse $t$-structures.
\end{enumerate}
\end{theorem}
\begin{proof}
\begin{enumerate}
\item
The four equalities are those established in the proof of \Cref{lem:perverse-t}, by \Cref{lem:local-duality} and transport of the standard $t$-structure along $\mathbb D_\omega$ as in \eqref{eq:transport}.
\item
Hearts of stable $\infty$-categorical $t$-structures are nerves of abelian categories \cite[Remark 1.2.1.12]{HA}; restriction of $\mathbb D_\omega$ gives the stated anti-equivalences of hearts.
\item
The forgetful functor to $\mathcal D^b_{\mathrm{coh}}(X)$ preserves the full subcategories of \Cref{def:dual-perversity} by construction, so it is $t$-exact.
\end{enumerate}
\end{proof}

\begin{remark}\label{rem:perverse-aside}
The $t$-exact forgetful functors of \Cref{thm:perverse}(3) do not identify $\Db(\Coh^C_X)$ with $\mathcal D(\mathrm{Cart}(\mathrm{Mod}(X),F_*))$ as in \cite[Theorem 6.9]{MWC}; that comparison is not used.
The equality $d_\omega(x)=-\dim\overline{\{x\}}$ is a separate normalization, treated in \Cref{cor:middle}.
\end{remark}

\begin{corollary}\label{cor:middle}
Suppose that, after shifting $\omega$ on each connected component, its local degree satisfies
\begin{equation}\label{eq:normalization}
 d_\omega(x)=-\dim\overline{\{x\}}\qquad(x\in X).
\end{equation}
Then \Cref{thm:perverse} holds for the middle-perversity formulas of \cite[Definition 5.2.1]{Bau}, namely
\begin{align}
 M\in{}^{\omega}\Db(\Coh^F_X)^{\leq0}
 &\quad\Longleftrightarrow\quad H^j(M_x)=0
               \quad\bigl(j>-\dim\overline{\{x\}},\ x\in X\bigr),\label{eq:middle-stalk}\\
 M\in{}^{\omega}\Db(\Coh^F_X)^{\geq0}
 &\quad\Longleftrightarrow\quad
 H^j_{\mathfrak m_x}(M_x)=0
               \quad\bigl(j<-\dim\overline{\{x\}},\ x\in X\bigr),\label{eq:middle-support}
\end{align}
and the same formulas with $\Coh^F_X$ replaced by $\Coh^C_X$.
\end{corollary}
\begin{proof}
Substitute \eqref{eq:normalization} in \eqref{eq:perv-stalk} and \eqref{eq:perv-support}.
\end{proof}

\begin{remark}[Why \eqref{eq:normalization} has to be added as an explicit condition]\label{rem:middle-counterexample}
Condition \eqref{eq:normalization} does not follow from Noetherianity, $F$-finiteness and existence of a unit dualizing complex.
Let
\[
 R=\F_p[[t]][x],\qquad X=\operatorname{Spec}R.
\]
This is a regular Noetherian ring of dimension two.
Its Frobenius is finite free on the $p$-basis
\[
 B=\{t^ix^j:0\leq i,j<p\}
\]
as an $R$-module.
Define an $R$-linear Cartier operator $\kappa:F_*R\to R$ on the unique expression of an element in this basis by
\[
 \kappa\Bigl(F_*\sum_{0\leq i,j<p}a_{ij}^{p}t^ix^j\Bigr)=a_{p-1,p-1}.
\]
We show that $(R,\kappa)$ is a unit Cartier module. In fact, the $R$-bilinear pairing
\[
 F_*R\times F_*R\to R,\qquad
 (F_*a,\,F_*b)\mapsto\kappa\bigl(F_*(ab)\bigr)
\]
is perfect: on basis elements, $\kappa(F_*(t^ix^j\cdot t^kx^\ell))$ equals $1$ if $i+k=p-1$ and $j+\ell=p-1$, and equals $0$ otherwise.
(If $i+k\geq p$ or $j+\ell\geq p$, rewrite the product in the basis $B$ by $F_*(t^p m)=t\,F_*(m)$ and $F_*(x^p m)=x\,F_*(m)$; the coefficient of $F_*(t^{p-1}x^{p-1})$ is still zero whenever $(i,j)$ and $(k,\ell)$ are not complementary.)
Thus with respect to $B$ the pairing matrix is a permutation matrix, and the induced map
\[
 F_*R\to\Hom_R(F_*R,R),\qquad
 F_*a\mapsto\bigl(F_*b\mapsto\kappa(F_*(ab))\bigr)
\]
is an isomorphism.
The complex $R[0]$ is dualizing since $R$ is regular of finite dimension \cite[\href{https://stacks.math.columbia.edu/tag/0AWX}{Tag 0AWX}]{SP}.
On a regular scheme the inclusion of unit Cartier modules is exact \cite[Lemma 3.4.11]{Bau}, so $R[0]$ is a unit dualizing complex.

The ideal $\mathfrak q=(tx-1)$ is maximal, with $R/\mathfrak q\simeq\F_p((t))$, and has height one.
The generic point specializes to $\mathfrak q$ in codimension one, while $-\dim\overline{\{x\}}$ changes from $-2$ to $0$.
It cannot be the local degree of a shifted dualizing complex, whose change along such a specialization is one \cite[\href{https://stacks.math.columbia.edu/tag/0AWF}{Tag 0AWF}]{SP}.
The catenarity argument in \cite[proof of Lemma 5.2.2]{Bau} therefore does not justify \eqref{eq:normalization} in that generality.
\end{remark}

\begin{corollary}\label{cor:crystal-t}
Each of $\KF(X)$ and $\KC(X)$ has its standard bounded $t$-structure, with
\[
 \KF(X)^{\heartsuit}=\Crys^F_X,\qquad
 \KC(X)^{\heartsuit}=\Crys^C_X.
\]
Transport along $\mathbb D^{\mathrm{crys}}_\omega$ gives bounded dual $t$-structures and anti-equivalences between their hearts and the opposite standard crystal hearts.
\end{corollary}
\begin{proof}
The Frobenius and Cartier cases are parallel.
By \Cref{def:nil-complex}, $\KF(X)$ (resp.\ $\KC(X)$) is the full subcategory of $\mathcal D(\QCrys^F_X)$ (resp.\ $\mathcal D(\QCrys^C_X)$) on objects with bounded cohomology in the essential image of $\Crys^F_X$ (resp.\ $\Crys^C_X$).
That essential image is a Serre subcategory \cite[Lemma 3.3.6]{Bau} on the Cartier side, and \cite[Proposition 3.4.2(b)(c)]{BP} on both abelian sides.
The standard truncations of $\mathcal D(\QCrys^F_X)$ and $\mathcal D(\QCrys^C_X)$ do not change cohomology objects in the remaining range.
The cohomology long exact sequence of a fiber sequence keeps those objects in a Serre subcategory of the heart, and boundedness is preserved, so the truncations preserve $\KF(X)$ and $\KC(X)$.
Stability of $\KF(X)$ and $\KC(X)$ is \Cref{lem:bounded-serre}.
On the homotopy categories, the restricted pairs of full subcategories $(\cdot)^{\leq0}$ and $(\cdot)^{\geq0}$ satisfy the three axioms of \cite[Definition 1.2.1.1]{HA}: the truncation fiber sequences of \cite[Remark 1.2.1.8]{HA} remain in $\KF(X)$ and $\KC(X)$.
A $t$-structure on a stable $\infty$-category is a $t$-structure on its homotopy category, as before \eqref{eq:transport}.
The hearts are those Serre subcategories.
The inclusions $\Db(\Crys^F_X)\to\KF(X)$ and $\Db(\Crys^C_X)\to\KC(X)$ of \Cref{prop:quotient-comparison} are $t$-exact for the standard $t$-structures, because the abelian inclusions are exact, and they are equivalences, so the hearts are $\Crys^F_X$ and $\Crys^C_X$.
Now apply \eqref{eq:transport} to \Cref{thm:crystals}.
\end{proof}

\section{Finite Zariski descent}\label{sec:descent}
In this section we work in the larger universe $\mathbb V$ of \Cref{subsec:conventions}: $\Catbig_\infty$ is an $\infty$-category in $\mathbb V$, and so is the arrow $\infty$-category $(\Catbig_\infty)^{\Delta^1}$ \cite[\href{https://kerodon.net/tag/0066}{Tag 0066}]{Kerodon}.
Write $\mathrm{End}(\Catbig_\infty)$ for the pullback
\[
\begin{tikzcd}
 \mathrm{End}(\Catbig_\infty)\ar[r]\ar[d] &
 (\Catbig_\infty)^{\Delta^1}\ar[d,"{(\ev_0,\ev_1)}"]\\
 \Catbig_\infty\ar[r,"{\Delta}"']&
 \Catbig_\infty\times\Catbig_\infty,
\end{tikzcd}
\]
formed in $\mathbb V$, i.e., the pullback of $(\ev_0,\ev_1)$ along the diagonal $\Delta$.
In terms of lax equalizers,
\[
 \mathrm{End}(\Catbig_\infty)=\LEq(\mathrm{id}_{\Catbig_\infty},\mathrm{id}_{\Catbig_\infty}).
\]
An object is a pair $(\cC,T)$ with $T:\cC\to\cC$.
A \emph{diagram of $\infty$-categories with endofunctors} is a functor $I\to\mathrm{End}(\Catbig_\infty)$.
Limits in $\mathrm{End}(\Catbig_\infty)$ are computed on underlying $\infty$-categories: the forgetful functor $\mathrm{End}(\Catbig_\infty)\to\Catbig_\infty$ preserves limits, and
\[
 \lim_i(\cC_i,T_i)\simeq\bigl(\lim_i\cC_i,\,\lim_i T_i\bigr).
\]

\begin{lemma}\label{lem:open-push}
Let $f:X\to Y$ be finite between Noetherian schemes, let $j_Y:V\to Y$ be a quasi-compact open immersion, and write $X_V=X\times_Y V$.
There is a natural equivalence
\[
 j_Y^*f_*\simeq f_{V,*}j_X^*
\]
in $\Fun(\Dqc(X),\Dqc(V))$.
\end{lemma}
\begin{proof}
By \Cref{lem:finite-push}, $f_*$ and $f_{V,*}$ are obtained by gluing restriction of scalars along finite ring maps on affines of $Y$ and of $V$, respectively.
Restriction along $j_Y$ therefore gives a comparison morphism
\[
 j_Y^*f_*\to f_{V,*}j_X^*
\]
in $\Fun(\Dqc(X),\Dqc(V))$.
On every affine open of $V$ both sides are restriction of scalars along the same finite ring map, so the comparison is an equivalence there via \Cref{lem:dqc-affine}.
Affine restrictions are jointly conservative (\Cref{lem:projections-conservative}), hence the comparison is an equivalence in $\Fun(\Dqc(X),\Dqc(V))$ \cite[\href{https://kerodon.net/tag/01DK}{Tag 01DK}]{Kerodon}.
\end{proof}

\begin{lemma}\label{lem:leq-limits}
Let $(\cC_i,T_i)$ be a diagram in $\mathrm{End}(\Catbig_\infty)$, and let $(\cC,T)$ be its limit.
Then
\[
 \Cart(\cC,T)\simeq\lim_i\Cart(\cC_i,T_i),\qquad
 \Frob(\cC,T)\simeq\lim_i\Frob(\cC_i,T_i).
\]
If in addition one is given a diagram of equivalences as in \Cref{prop:reversal}, applying \Cref{prop:reversal} to the limit yields the limit of the levelwise applications.
\end{lemma}
\begin{proof}
For an $\infty$-category $\cD$, the simplicial set $(\cD)^{\Delta^1}$ is an $\infty$-category \cite[\href{https://kerodon.net/tag/0066}{Tag 0066}]{Kerodon}; equivalently \cite[Proposition 1.2.7.3]{HTT}.
The category of $\infty$-categories is cartesian closed \cite[\href{https://kerodon.net/tag/05UH}{Tag 05UH}]{Kerodon}.
The exponential law of simplicial sets \cite[\href{https://kerodon.net/tag/0061}{Tag 0061}]{Kerodon} is an isomorphism $\Fun(\cC\times\Delta^1,\cD)\cong\Fun(\cC,(\cD)^{\Delta^1})$ of simplicial sets, hence of $\infty$-categories when $\cD$ is an $\infty$-category.
Mapping spaces in $\Catbig_\infty$ are the maximal Kan complexes of $\Fun$ \cite[Definition 3.0.0.1]{HTT}.
This isomorphism therefore induces an equivalence of mapping spaces, and the adjunction $(-\times\Delta^1)\dashv(-)^{\Delta^1}$ of $\infty$-functors on $\Catbig_\infty$ follows from the mapping-space criterion \cite[Proposition 5.2.2.8]{HTT}.
Right adjoints of $\infty$-functors preserve limits \cite[Proposition 5.2.3.5]{HTT}.
A limit functor is right adjoint to the diagonal \cite[\href{https://kerodon.net/tag/02JL}{Tag 02JL}]{Kerodon}, so products and pullbacks preserve limits.
The $\infty$-category $\Catbig_\infty$ has small limits \cite[Corollary 4.2.4.8; Remark 3.0.0.5]{HTT}.
Apply these facts to the pullback of \Cref{def:leq}.
The functor $(-)^{\op}:\Catbig_\infty\to\Catbig_\infty$ is an involution, as in the proof of \Cref{prop:reversal}, so it preserves limits of $\infty$-categories.
A levelwise equivalence in $\Fun(I,\Catbig_\infty)$ is an equivalence \cite[\href{https://kerodon.net/tag/01DK}{Tag 01DK}]{Kerodon}.
The limit of a diagram of data $(D_i,\gamma_i)$ as in \Cref{prop:reversal} is a pair $(\lim D_i,\lim\gamma_i)$ with the limit natural equivalences, so \Cref{prop:reversal} applies to the limit pair.
The resulting equivalence is then the limit of the levelwise equivalences, because $\Cart$ and $\Frob$ preserve the limits already identified.
\end{proof}

\begin{theorem}\label{thm:descent}
Let $X$ be Noetherian and $F$-finite, and let $\{U_i\}_{i=1}^r$ be a finite Zariski open cover.
Write $U_\bullet$ for its \v Cech nerve.
\begin{enumerate}
\item
Restriction induces equivalences
\begin{align*}
 \Dbc(X)&\simeq\lim_{[n]\in\Delta}\Dbc(U_n),\\
 \Cart(\Dbc(X),F_*)&\simeq
       \lim_{[n]\in\Delta}\Cart(\Dbc(U_n),F_*),\\
 \Frob(\Dbc(X),F_*)&\simeq
       \lim_{[n]\in\Delta}\Frob(\Dbc(U_n),F_*).
\end{align*}
\item
For fixed unit dualizing data $(\omega,u)$, the equivalence of \Cref{thm:intrinsic} is the limit of the corresponding equivalences on $U_\bullet$.
\item
When $X$ has affine diagonal, \Cref{thm:comparison} identifies $\Db(\Coh^F_X)$ with $\Frob(\Dbc(X),F_*)$, hence with $\lim_{[n]}\Frob(\Dbc(U_n),F_*)$, and likewise on the Cartier side.
\end{enumerate}
\end{theorem}
\begin{proof}
\begin{enumerate}
\item
Quasi-coherent $\infty$-categories satisfy Zariski descent by \Cref{lem:dqc-zariski}: there is an equivalence $\Dqc(X)\simeq\lim_{[n]}\Dqc(U_n)$ in $\Catbig_\infty$.
Restriction $\Dqc(X)\to\Dqc(U_n)$ preserves the full subcategory $\Dbc$, so there is a cone, and therefore a functor
\[
 \Phi:\Dbc(X)\longrightarrow\lim_{[n]\in\Delta}\Dbc(U_n)
\]
in $\Cat_\infty$.
By \Cref{def:dqc}, the inclusion $\Dbc\subset\Dqc$ is that of a full subcategory, hence fully faithful.
Mapping spaces in a limit are limits of mapping spaces, as in \Cref{lem:projections-conservative}.
Hence $\Phi$ is fully faithful.

For essential surjectivity, an object of the \v Cech limit is an object $M\in\Dqc(X)$ whose restriction to each $U_n$, and in particular to each $U_i$, lies in $\Dbc(U_i)$.
Refine each $U_i$ by a finite affine cover.
Coherence of cohomology sheaves of $M$ may be tested on a finite affine cover \cite[\href{https://stacks.math.columbia.edu/tag/01XZ}{Tag 01XZ}]{SP}.
Vanishing of a sheaf is local on an open cover.
Since the cover $\{U_i\}$ is finite, if the restrictions have cohomology in $[a_i,b_i]$ then the cohomology of $M$ is globally in $[\min_i a_i,\max_i b_i]$.
Thus $M\in\Dbc(X)$, and $\Phi$ is an equivalence.

The \v Cech diagram $[n]\mapsto\Dqc(U_n)$ is the one constructed in \Cref{lem:dqc-zariski}.
Each simplicial operator $\alpha:[m]\to[n]$ is sent to a coproduct of open immersions $U_n\to U_m$.
On each summand, \Cref{lem:open-push} gives a morphism $j^*F_{W,*}\to F_{V,*}j^*$ in $\Fun(\Dqc(W),\Dqc(V))$, which is an equivalence.
On a finite coproduct, $F_*$ is the product of the summands by \Cref{lem:dqc-coproduct} and the affine formula for restriction of scalars.
These equivalences are compatible with composition of operators: after restriction to every affine open, both composites are restriction of scalars along the same finite ring map, so \cite[\href{https://kerodon.net/tag/01DK}{Tag 01DK}]{Kerodon} applies in $\Fun$.
Thus the family of endofunctors $F_{U_n,*}$ is a morphism of the \v Cech diagram in $\Catbig_\infty$, i.e., a functor $\Delta\to\mathrm{End}(\Catbig_\infty)$ whose value at $[n]$ is $(\Dqc(U_n),F_{U_n,*})$.
The functor $F_*$ preserves $\Dbc$ by the affine case \cite[\href{https://stacks.math.columbia.edu/tag/00GJ}{Tag 00GJ}]{SP}, locality of coherence of cohomology sheaves as already used for $\Phi$, and locality of cohomological vanishing on a finite cover, and so does restriction, so the diagram restricts to $\Dbc$.
Its limit cone is $(\Dbc(X),F_*)$, by the equivalence $\Phi$ and the identification of limits in $\mathrm{End}(\Catbig_\infty)$.
Apply \Cref{lem:leq-limits} to obtain the other two equivalences.
\item
The same simplicial operators act on the data of \Cref{prop:reversal} as follows.
For each open immersion $j$, the affine identification of internal Hom in \Cref{lem:dqc-properties,lem:dqc-affine} together with \cite[\href{https://stacks.math.columbia.edu/tag/0A6A}{Tag 0A6A(3)}]{SP} gives a morphism $j^*D_{\omega_X}\to D_{\omega_U}(j^*)^{\op}$ in $\Fun(\Dbc(X)^{\op},\Dqc(U))$, which is an equivalence.
\Cref{lem:open-rhom} gives $j^*F^!\omega_X\simeq F^!j^*\omega_X$ on dualizing complexes, so the unit $u$ restricts.
The natural equivalence $\gamma$ of \Cref{prop:finite-duality} is natural in open immersions by \Cref{lem:finite-push,lem:open-push,lem:open-rhom,lem:projection}, so it restricts to a morphism in $\Fun(\Dbc(U)^{\op},\Dbc(U))$.
Compatibility with composition of \v Cech operators is again detected on affines.
This is a diagram of data as in \Cref{prop:reversal} on $U_\bullet$, in the sense of \Cref{lem:leq-limits}.
That lemma identifies the global equivalence of \Cref{thm:intrinsic} with the limit of the local equivalences.
\item
Assume $X$ has affine diagonal.
Opens and finite coproducts of schemes with affine diagonal still have affine diagonal, so \Cref{thm:comparison} applies to $X$ and to each $U_n$.
Composing the global comparison with the limit equivalences of \textup{(1)} gives the identifications in \textup{(3)}.
\end{enumerate}
\end{proof}

\begin{remark}[Bounds and completeness]\label{rem:descent-bounds}
On the non-quasi-compact scheme $\coprod_{n\geq0}\operatorname{Spec}k$, the object equal to $k[n]$ on the $n$th component is bounded coherent on each member of the component cover, but has no global finite cohomological bound.
Thus global boundedness cannot be dropped from the finite-cover argument.
The unbounded derived Frobenius comparison has a separate completeness step \cite[Lemma 5.1; Proposition 5.2; Corollary 5.5; Theorem 5.7]{MWF}.
The additional descent issue in the method of Mattis--Wei\ss\ is explained in \cite[Remark 4.3]{MWF}.
\end{remark}

\section{Ind--Pro duality}\label{sec:indpro}
Recall that for a $\mathbb U$-small stable $\infty$-category $\cC$, the $\infty$-category of pro-objects in $\cC$ is defined by
$ \Pro(\cC):=\Ind(\cC^{\op})^{\op}$.
The completion is taken in the fixed universe $\mathbb V$ of \Cref{subsec:conventions}.
If $\cC$ is small and stable, then $\Ind(\cC)$ is stable \cite[Proposition 1.1.3.6]{HA}.
Essentially small sources are replaced by small equivalent $\infty$-categories as in \Cref{subsec:conventions}.


\begin{corollary}\label{cor:indpro-modules}
Let $X$ be Noetherian and $F$-finite, and let $\omega$ be a unit dualizing complex.
\begin{enumerate}
\item
The equivalence $\Pro(\cC)^{\op}\simeq\Ind(\cD)$ yields
\begin{align}
 \Pro(\Dbc(X))^{\op}
  &\simeq
 \Ind(\Dbc(X)),\label{eq:indpro-dbc}\\
 \Pro\bigl(\Frob(\Dbc(X),F_*)\bigr)^{\op}
  &\simeq
 \Ind\bigl(\Cart(\Dbc(X),F_*)\bigr),\label{eq:indpro-leq}\\
 \Pro\bigl(\Db(\Coh^F_X)\bigr)^{\op}
  &\simeq
 \Ind\bigl(\Db(\Coh^C_X)\bigr),\label{eq:indpro-modules-dc}\\
 \Pro\bigl(\cN_F(X)\bigr)^{\op}
  &\simeq
 \Ind\bigl(\cN_C(X)\bigr),\label{eq:indpro-nil}\\
 \Pro\bigl(\KF(X)\bigr)^{\op}
  &\simeq
 \Ind\bigl(\KC(X)\bigr).\label{eq:indpro-k}
\end{align}
\item
\Cref{prop:quotient-comparison} identifies $\KF(X)$ with $\Db(\Coh^F_X)/\cN_F(X)$ and with $\Db(\Crys^F_X)$, and likewise on the Cartier side, hence identifies \eqref{eq:indpro-k} with
\begin{align}
 \Pro\bigl(\Db(\Coh^F_X)/\cN_F(X)\bigr)^{\op}
  &\simeq
 \Ind\bigl(\Db(\Coh^C_X)/\cN_C(X)\bigr),\label{eq:indpro-quot}\\
 \Pro\bigl(\Db(\Crys^F_X)\bigr)^{\op}
  &\simeq
 \Ind\bigl(\Db(\Crys^C_X)\bigr).\label{eq:indpro-crys}
\end{align}
\item
If in addition $X$ has affine diagonal, \Cref{thm:comparison} identifies $\Frob(\Dbc(X),F_*)$ with $\Db(\Coh^F_X)$ and $\Cart(\Dbc(X),F_*)$ with $\Db(\Coh^C_X)$, hence identifies \eqref{eq:indpro-leq} and \eqref{eq:indpro-modules-dc}.
\end{enumerate}
\end{corollary}
\begin{proof}
\begin{enumerate}
\item
The dualities of \textup{(1)} are the exact equivalences of small stable $\infty$-categories in \Cref{lem:grothendieck,thm:intrinsic,thm:modules,thm:crystals}.
For \Cref{thm:intrinsic}, the pair $(\omega,u)$ is the affine reconstruction recorded in \Cref{rem:unit-choice}.
\item
The identifications of $\KF(X)$ with $\Db(\Coh^F_X)/\cN_F(X)$ and with $\Db(\Crys^F_X)$, and likewise on the Cartier side, are \Cref{prop:quotient-comparison}; hence that proposition identifies \eqref{eq:indpro-k} with \eqref{eq:indpro-quot} and \eqref{eq:indpro-crys}.
\item
Assume $X$ has affine diagonal.
\Cref{thm:comparison} gives $\Db(\Coh^F_X)\simeq\Frob(\Dbc(X),F_*)$ and $\Db(\Coh^C_X)\simeq\Cart(\Dbc(X),F_*)$, hence identifies \eqref{eq:indpro-leq} and \eqref{eq:indpro-modules-dc}.
\end{enumerate}
\end{proof}

\begin{remark}[Geometric versus coefficient Ind--Pro]\label{rem:chow-indpro}
Chowdhury extends a presentable six-functor formalism on a geometric setup $(C,E)$ to certain Ind- and Pro-objects of $C$ itself \cite[Definitions 4.1.6 and 4.1.8; Theorem 4.2.1]{Chow-IndPro}.
For a pro-object $\{X_k\}$ in that restricted class one has $D^{\mathrm{pro}}(\mathcal X)\simeq\operatorname{colim}^* D(X_k)$ in presentable $\infty$-categories \cite[Remark 4.2.2(2)]{Chow-IndPro}.
That construction changes the geometric input, not the coefficient category of a fixed $X$.
The equivalence $\Pro(\cC)^{\op}\simeq\Ind(\cD)$ is the opposite: 
it extends a duality already constructed on small coefficient categories of a fixed $X$ (\Cref{cor:indpro-modules}), by the universal property of $\Ind$ \cite[Proposition 5.3.5.10]{HTT}.
The two procedures do not imply each other.
In particular, Chowdhury's target is presentable, while a Pro-category of coefficients need not be presentable.
\end{remark}

\end{document}